\documentclass[12pt,oneside]{amsart}
\usepackage[utf8]{inputenc}
\usepackage[english]{babel}
\usepackage{amsfonts}
\usepackage{amssymb, amsmath}
\usepackage{amsthm}
\usepackage{graphicx}
\usepackage{comment}
\usepackage{pgf}
\usepackage{graphics,epsfig, subfigure}
\usepackage{url}
\usepackage{srcltx}
\usepackage{hyperref}

 \theoremstyle{plain}
      \newtheorem{theorem}{Theorem}
      \newtheorem{lemma}{Lemma}[section]
      
      \newtheorem{problem}{Problem}
      
      \theoremstyle{definition}

      \title
[Projections and sections of polytopes]
{On polytopes with congruent projections or sections}

\author{Sergii Myroshnychenko}
\address{Department of Mathematics, Kent State University,
Kent, OH 44242, USA} \email{smyroshn@kent.edu}

\author{Dmitry Ryabogin}
\address{Department of Mathematics, Kent State University,
Kent, OH 44242, USA} \email{ryabogin@math.kent.edu}

\thanks{The authors are supported in
part by U.S.~National Science Foundation Grant DMS-1600753}

\keywords{Sections and projections of convex bodies}

\begin{document}

\begin{abstract}
Let $2\le k\le d-1$ and let $P$ and $Q$ be two convex polytopes in ${\mathbb E^d}$. Assume that their projections, $P|H$, $Q|H$, onto every $k$-dimensional subspace $H$,  are congruent. In this paper we show  that  $P$ and $Q$ or $P$ and $-Q$ are translates of each other. We also prove an analogous result for sections by showing that $P=Q$ or $P=-Q$, provided the  polytopes  contain the  origin in their interior and their sections, $P \cap H$, $Q \cap H$, by every $k$-dimensional subspace $H$, are congruent.

\end{abstract}

\maketitle

\section{Introduction}

In this paper we address the following problems (cf., for example, \cite[ Problem 3.2, p. 125 and Problem 7.3, p. 289]{Ga}).
\begin{problem}\label{pr1}
Let $2\le k\le d-1$. Assume that $K$ and $L$ are convex bodies in ${\mathbb E}^d$ such that the projections $K|H$ and $L|H$ are congruent for all $H\in {\mathcal G}(d,k)$. Is $K$ a translate of $\pm L$?
\end{problem}
\begin{problem}\label{pr2}
Let $2\le k\le d-1$. Assume $K$ and $L$ are star bodies in ${\mathbb E}^d$ such that the sections $K\cap H$ and $L\cap H$ are congruent for all $H\in {\mathcal G}(d,k)$. Is $K=\pm L$?
\end{problem}
Here we say that $K|H$, the projection of $K$ onto $H$,  is congruent to $L|H$ if there exists   an orthogonal transformation  $\varphi\in O(k,H)$ in $H$ such that $\varphi(K|H)$ is  a translate of  $L|H$; ${\mathcal G}(d,k)$ stands for the Grassmann manifold of all $k$-dimensional subspaces in ${\mathbb E^d}$.

Golubyatnikov has obtained some partial answers to Problem \ref{pr1} in the case of {\it direct} rigid motions, i.e., when the orthogonal group $O(k)$ is replaced by the special orthogonal group $SO(k)$. In particular, he proved that the answer is affirmative, provided $k=2$ and none of the projections of the bodies have  symmetries with respect to rotations (in other words, the only direct rigid motion taking $K|H$, $L|H$ onto themselves is the identity).
We refer the reader to \cite{Go}, \cite{Ga} (pp. $100-110$), \cite{Ha} (pp. $126-127$), \cite{R1}, \cite{R2}, \cite{R3} and \cite{ACR}, for history and partial results related to  the above problems.

In this paper we give an affirmative answer to Problems \ref{pr1} and \ref{pr2} in the class of convex polytopes. Our first result is

\begin{theorem}\label{th1}
Let $2\le k\le d-1$ and let $P$ and $Q$ be two convex polytopes in ${\mathbb E^d}$ such that their projections $P|H$, $Q|H$, onto every $k$-dimensional subspace $H$,  are congruent. Then there exists $b\in {\mathbb E^d}$ such that $P=Q+b$ or $P=-Q+b$.
\end{theorem}

If   projections of the bodies are {\it directly} congruent,  Golubyatnikov  proved Theorem \ref{th1} in the case $k=2$ (the result is contained in the proof of Theorem 2.1.2, p. 19; cf. \cite{Go}, p. 17, Lemma 2.1.4), and in the case $k=3$ under the additional assumption that none of the  $3$-dimensional projections have  rigid motion symmetries (cf. \cite{Go}, p. 48, Theorem 3.2.1);  see also \cite{ACR}, where the assumption on symmetries was relaxed.

Our second result is

\begin{theorem}\label{th2}
Let $2\le k\le d-1$ and let $P$ and $Q$ be two convex polytopes in ${\mathbb E^d}$  containing the origin in their interior. Assume  that their sections, $P \cap H$, $Q \cap H$, by every $k$-dimensional subspace $H$,  are congruent. Then $P=Q$ or $P=-Q$.
\end{theorem}

If   sections of the bodies are {\it directly} congruent   Theorem \ref{th2}  is known in the case $k=2$ (cf. \cite{AC}), and in the case $k=3$ (cf. \cite{ACR}) under the additional assumptions that
diameters of one of the polytopes contain the origin and certain sections related to the diameters
 have no rigid motion symmetries.

We remark that in  both Theorems it is enough to assume that only one of the bodies is a convex polytope. This follows from a result of Klee  \cite{Kl}, who showed that a set in $\mathbb{E}^d$ must be a convex polytope, provided all of its projections on $k$-dimensional subspaces are convex polytopes. It follows by duality (see  \cite{Ga}, formula (0.38), p.  22) that if all sections of a convex body containing the origin in its interior are polytopes, then the body is a polytope.

The paper is organized as follows. In the next section we recall some definitions for convenience of the reader. We  also  prove some auxiliary Lemmata that will be used in Sections 3 and 4, where we prove Theorems \ref{th1} and \ref{th2}.

\vspace{0.5cm}

{\bf Acknowledgements:} We would like to thank Alexander Fish for useful discussions. We are indebted to the referee for their suggestions that considerably improved the paper.

\subsection{Notation}
We denote by $S^{d-1}=\{x\in \mathbb{E}^d:\,|x|=1\}$ the set of all unit vectors in the Euclidean space $\mathbb{E}^d$, and by $S^{d-1}_{\zeta}$$=\{x\in S^{d-1}$$:\,x\cdot \zeta\ < 0\}$ the open hemisphere in the direction $\zeta\in S^{d-1}$; $O$ stands for the origin of $\mathbb{E}^d$. We write that line $l_1$ is parallel to line $l_2$ as $l_1 \parallel l_2$.
For any unit vector $\xi \in S^{d-1}$ we let $\xi^{\perp}$ to be the orthogonal complement of $\xi$ in $\mathbb{E}^{d}$, i.e. the set of all $x \in \mathbb{E}^{d}$ such that $x \cdot \xi = 0$; here $x \cdot \xi$ stands for  a usual scalar product of $x$ and $\xi\in \mathbb{E}^d$. The notation for the orthogonal group $O(k)$ and the special orthogonal group $SO(k)$ is standard. The notation  $\varphi_{\xi}\in O(d-1,\xi^{\perp})$ will be used for
 an orthogonal transformation acting in $\xi^{\perp}$. We agree to denote   by $xy$  the closed interval connecting $x$ and $y$, i.e. all points of the form $ty+(1-t)x, t \in [0,1]$; the shortest arc of the unit circle joining the points $x$ and $y$ on $S^{d-1}$ will be denoted by $[xy]$; $span(a_1, a_2,\dots, a_m)$ stands for the $m$-dimensional subspace that is a linear span of the linearly independent  vectors $a_1,\dots, a_m$;
 $span(A)$ is the linear subspace of smallest dimension containing a set $A  \subset \mathbb{E}^d$.
Also, for any set $A \subset \mathbb{E}^d$ the notation  $A_{\xi}=A|\xi^{\perp}$ is used for the projection of $A$ onto $\xi^{\perp}$. The {\it shadow boundary} of a convex polytope $P$ in direction $\xi$ will be denoted by $\partial_{\xi}P$, i.e., $\partial_{\xi}P=\{x\in P:\, x_{\xi} \in  \partial P_{\xi}\}$;
here $\partial P_{\xi}$ stands for the boundary of the projection $P_{\xi}$. Given a set $B$, $\textrm{conv}(B)$ denotes the smallest convex set containing $B$.

\section{Definitions and Auxiliary Results}

A set $A\subset {\mathbb E^d}$ is convex if for any two points $x$ and $y$ in $A$ the closed line segment $xy$ joining them is in $A$. A convex {\it body} $K\subset {\mathbb E^d}$ is a compact convex set with a non-empty interior (with respect to ${\mathbb E^d}$). A {\it convex polytope} $P\subset {\mathbb E^d}$ is a convex body that is the convex hull of finitely many points (called the vertices of $P$).

Fix $1\le k\le d$. We say that a convex set $B$ is a {\it $k$-dimensional convex polytope}  if there exists an affine $k$-dimensional subspace $M$ of ${\mathbb E^d}$ such that $B\subset M$
and  $B$ is a convex polytope relative to $M$ (i.e., the interior of $B$ is taken with respect to $M$).

We will say that a  subset $D$ of an open hemisphere $S^{d-1}_{\zeta}$, $z\eta\in S^{n-1}$, is {\it geodesically convex}, if for any two points $x$ and $y$ in $D$ the arc of the unit circle  $[xy]$ joining them is in $D$.

We define  a {\it rigid motion} $T_{\xi}$ acting in $\xi^{\perp}$, $ T_{\xi}: \xi^{\perp} \to \xi^{\perp}$,  as  a composition of an orthogonal transformation $\varphi_{\xi}$
and a translation by $a_{\xi}\in \xi^{\perp}$, $T_{\xi}(x)=\varphi_{\xi}(x)+a_{\xi}$ for all $x\in\xi^{\perp}$.

 The supporting hyperplane  $G$  to a convex body $K$ is defined as the hyperplane having common points with $K$ and such that
$K$ lies in one of the two closed half-spaces with the boundary $G$. The {\it support function} of a convex body  is defined as
$h_K(\xi)=\max\limits_{x\in K} x\cdot \xi$ for all $\xi\in {\mathbb E^d}$.

If $L$ is a convex body containing the origin in the interior, its
{\it radial function}   is defined as $\rho_L(\xi)=\max\,\{\lambda>0:\,\lambda \xi\in L\}$ for all $\xi\in {\mathbb E^d}$.
Since $h_K$ and $\rho_K$ are homogeneous functions of degrees $1$ and $-1$ correspondingly, it is enough to consider their values for $\xi\in S^{d-1}$, where both functions are continuous.

Our first auxiliary result will be used in the proof of Theorem \ref{th1}.
\begin{lemma}\label{length}
Consider four points $A,B,C,D \in \mathbb{E}^d$, such that for an open set $U \subset S^{d-1}$ of directions $\xi$, $|A_{\xi}B_{\xi}| = |C_{\xi}D_{\xi}| \neq 0$. Then the intervals $AB$ and $CD$ are parallel and have equal length.
\end{lemma}

\begin{proof}
Since any parallel translation preserves the length of the projections  of the intervals, we may assume that $A = C= O$. Let $a$ and $b  $ be the unit vectors of directions of $AB$ and $CD$, and $t$ and $s$ be their lengths. In this case, the projections of $B$ and $D$ onto $\xi^{\perp}$ can be found as
$$
B_{\xi} = t a- (\xi \cdot ta) \xi \quad \textrm{and} \quad D_{\xi} = sb - (\xi \cdot sb) \xi, \qquad \forall \xi \in U.
$$
Our goal is to prove that $t = s$ and $a = b$ or $a = -b$.

The condition on the equality of the lengths of projections can be written as
$$
t|a - (\xi \cdot a) \xi| = s|b - (\xi \cdot b) \xi| \qquad \forall \xi \in U,
$$
or
$$
t^2-s^2=t^2(\xi\cdot a)^2-s^2(\xi\cdot b)^2=(\xi\cdot(ta-sb))(\xi\cdot(ta+sb))\qquad \forall \xi \in U.
$$

 We claim   that the right-hand side of the above identity is not  an identically  constant function of $\xi$ on the spherical cap $ U$, unless
$ta-sb=0$ or $ta+sb=0$. Indeed,  if the claim is false, then  dividing the above identity by the lengths  of vectors $ta-sb$, $ta+sb$, we see that the function
$$
f(\xi)=(\xi\cdot v)(\xi\cdot w), \quad v=\frac{ta-sb}{|ta-sb|},\quad w=\frac{ta+sb}{|ta+sb|},
$$
 is identically constant (equal to $\frac{t^2-s^2}{|ta-sb||ta+sb|}$) on $U$.  For $c_1, c_2 \in (-1,1)$ consider  the level sets
$$
L_v(c_1)=\{\xi\in S^{d-1}:\,\xi\cdot v=c_1\},\quad L_w(c_2)=\{\xi\in S^{d-1}:\,\xi\cdot w=c_2\},
$$
of functions
$\xi\to\xi\cdot v$ and $\xi\to \xi\cdot w$. It is clear that  $L_v(c_1)$, $L_w(c_2)$  are the corresponding $(d-2)$-dimensional sub-spheres of $S^{d-1}$. Since $v\neq \pm w$,  $L_v(c_1)\neq L_w(c_2)$ $\forall c_1$, $c_2\in(-1,1)$, and $L_v(c_1)\cap L_w(c_2)$ is
(at most)  an $(d-3)$-dimensional subsphere of $S^{d-1}$.  Taking $c_1, c_2\in (-1,1)$ such that $L_v(c_1)\cap L_w(c_2)\cap U\neq \emptyset$, and  changing $\xi$ in $U\cap (L_v(c_1)\setminus L_w(c_2))$, we see that $\xi\cdot v$ remains constant while $\xi\cdot w$ does not. Thus,  $f(\xi)$ is not constant on $U$, and we obtain a contradiction. The claim is proven, and the proof of the Lemma is finished.
\end{proof}

The next statement shows that one may disregard the set of all directions  for which at least two facets of the projections are orthogonal. Here we say that two affine subspaces  of co-dimension one are orthogonal if  their normal vectors are orthogonal; two facets of a polytope  are orthogonal if the affine subspaces containing them are.

\begin{lemma}\label{perp}
Let $\alpha$ and $\beta$ be two non-parallel $(d-2)$-dimensional subspaces in $\mathbb{E}^d$, such that $\dim (\alpha \cap \beta) = d-3$. Then the set of directions $\xi \in S^{d-1}\setminus (\alpha\cup \beta)$, for which  $\alpha_{\xi}$ is orthogonal to $ \beta_{\xi}$, is a nowhere dense  subset of $S^{d-1}$.
\end{lemma}

\begin{proof}
Assume that $\alpha \cap \beta = span \{ c_1, c_2,\ldots, c_{d-3}\}$ and $\alpha = span \{a, c_1,\ldots,c_{d-3} \}$, $\beta = span \{b, c_1,\ldots, c_{d-3} \}$, where
$a$, $b$, and $c_j$ are linearly independent vectors in ${\mathbb E^d}$, $j=1,\dots, d-3$, $a\neq b$.
By condition of the lemma, we have $\alpha_{\xi} = $$span \{a_{\xi}, (c_1)_{\xi}$, $\ldots, (c_{d-3})_{\xi} \}$, $\beta_{\xi} = span \{b_{\xi}, (c_1)_{\xi},\ldots, (c_{d-3})_{\xi} \}$,
 $\alpha_{\xi} \cap \beta_{\xi} = span \{(c_1)_{\xi}, \ldots, (c_{d-3})_{\xi} \}$. Here
 $(c_{i})_{\xi} = c_{i} - (c_i \cdot \xi) \xi$, $i=1$, $\dots$, $d-3$, and $a_{\xi} = a - ( a \cdot \xi) \xi$, $b_{\xi} = b - ( b\cdot \xi) \xi$.

Let $n_{\alpha_{\xi}}$ and $n_{\beta_{\xi}}$ be the  normal vectors to $\alpha_{\xi}$ and $\beta_{\xi}$ in $\xi^{\perp}$.
 If $\alpha_{\xi} \perp \beta_{\xi}$, then $n_{\alpha_{\xi}} \perp n_{\beta_{\xi}}$, and, in the three-dimensional case,
 the condition of orthogonality $n_{\alpha_{\xi}} \cdot n_{\beta_{\xi}}=0$ can be written as
 $$
a_{\xi} \cdot b_{\xi} = (a-\xi (a\cdot\xi))\cdot (b - \xi (b\cdot \xi))= a\cdot b-(a\cdot \xi)( b\cdot \xi)= 0.
$$
 To write out the condition  in the case $d\ge 4$, we  consider a linear transformation $A_{\xi} \in O(d)$, such that $A_{\xi} \xi =e_d=(0,\dots,0,1)$. For all $\xi\in S^{d-1}$, $\xi\neq e_d$, we identify $A_{\xi}$ with the corresponding matrix
$$
A_{\xi} = \left( \begin{array}{ccccc}
-1+\frac{\xi_1^2}{1-\xi_d} & \frac{\xi_1\xi_2}{1-\xi_d} & \ldots & \frac{\xi_1 \xi_{d-1}}{1-\xi_d} & -\xi_1\\
-\frac{\xi_1\xi_2}{1-\xi_d} & 1-\frac{\xi_2^2}{1-\xi_d}& \ldots & -\frac{\xi_2 \xi_{d-1}}{1-\xi_d} & \xi_2\\
\ldots&\ldots & \ldots & \ldots & \ldots \\
-\frac{\xi_1\xi_{d-1}}{1-\xi_d} & -\frac{\xi_{d-1} \xi_2}{1-\xi_d}& \ldots & 1-\frac{\xi_{d-1}^2}{1-\xi_d} & \xi_{d-1}\\
\xi_1 & \xi_2 & \ldots & \xi_{d-1} & \xi_d\\
\end{array} \right),
$$
(if we treat the $i$-th row  $r_i$ of $A_{\xi}$  as a $d$-dimensional vector, then $r_i \cdot r_j =\delta_{ij}$, $i,j$$=$$1$,  $\dots$, $d$). If $\xi=e_d$ we put   $A_{\xi} = I$.

Consider now  the vectors $\tilde{a}=\tilde{a}({\xi}) = A_{\xi} a_{\xi}, \tilde{b}=\tilde{b}({\xi}) = A_{\xi} b_{\xi}, \tilde{c_i}=\tilde{c}_i({\xi}) = A_{\xi} (c_i)_{\xi}, \tilde{n}_{\alpha_{\xi}} = A_{\xi} n_{\alpha_{\xi}}, \tilde{n}_{\beta_{\xi}} = A_{\xi} n_{\beta_{\xi}}$. The condition of the normals being orthogonal is
\begin{equation}\label{zero}
 \tilde{n}_{\alpha_{\xi}} \cdot \tilde{n}_{\beta_{\xi}} = n_{\alpha_{\xi}} \cdot n_{\beta_{\xi}}=0 \qquad \xi\in S^{d-1}\setminus (\alpha\cup \beta),
\end{equation}
where the normal vectors $\tilde{n}_{\alpha_{\xi}}, \tilde{n}_{\beta_{\xi}}\in e_d^{\perp}$ can be found as the generalized vector products,
$$
\tilde{n}_{\alpha_{\xi}} = \left| \begin{array}{cccccc}
e_1 & e_2 & e_3 & \ldots & e_{d-1}\\
\tilde{a}^1 & \tilde{a}^2 & \tilde{a}^3 & \ldots & \tilde{a}^{d-1}\\
\tilde{c}^{1}_{1} & \tilde{c}_{1}^{2} & \tilde{c}_{1}^{3} & \ldots & \tilde{c}_{1}^{d-1}\\
\ldots & \ldots  & \ldots  & \ldots & \ldots \\
\tilde{c}_{d-3}^{1}& \tilde{c}_{d-3}^{2} & \tilde{c}_{d-3}^{3} & \ldots & \tilde{c}_{d-3}^{d-1}\end{array} \right|,
$$
$$\tilde{n}_{\beta_{\xi}} = \left| \begin{array}{cccccc}
e_1 & e_2 & e_3 & \ldots & e_{d-1}\\
\tilde{b}^1 & \tilde{b}^2 & \tilde{b}^3& \ldots & \tilde{b}^{d-1}\\
\tilde{c}_{1}^{1}& \tilde{c}_{1}^{2} & \tilde{c}_{1}^{3} & \ldots & \tilde{c}_{1}^{d-1}\\
\ldots & \ldots  & \ldots  & \ldots & \ldots \\
\tilde{c}_{d-3}^{1}& \tilde{c}_{d-3}^{2}& \tilde{c}_{d-3}^{3} & \ldots & \tilde{c}_{d-3}^{d-1} \end{array} \right|.
$$
Here $\tilde{a}^j, \tilde{b}^{j}$ stand for the $j$-th coordinate of the vectors $\tilde{a}$ and $\tilde{b}$,  $\tilde{c}_{i}^{j}$ is the $j$-th coordinate of the vector $\tilde{c}_i$,
and $e_j$, $j=1$, $\dots$, $d-1$, are the vectors of the standard basis in $e_d^{\perp}$.

Notice that the $j$-th coordinate of the normal $\tilde{n}_{\alpha_{\xi}}$ is a rational function $R_j^{\alpha}(\xi)= \frac{P_j^{\alpha}(\xi)}{(1-\xi_d)^{d-2}}$ of  the coordinates of $\xi$,  where  $P^{\alpha}_j(\xi)$ is a polynomial in coordinates of $\xi$. Similarly for $\tilde{n}_{\beta_{\xi}}$, i.e.
$$
\tilde{n}_{\alpha_{\xi}} = (R^{\alpha}_1(\xi),\ldots,R^{\alpha}_{d-1}(\xi)), \quad \tilde{n}_{\beta_{\xi}} = (R^{\beta}_1(\xi),\ldots,R^{\beta}_{d-1}(\xi)).
$$
Since we can choose the standard basis in ${\mathbb E^d}$ in such a way that  $e_d\in n_{\alpha}^{\perp}\cup n_{\beta}^{\perp}$, we can assume that  $\xi_d\neq 1$.

Multiplying each vector $n_{\alpha_{\xi}}$, $n_{\beta_{\xi}}$ by $(1-\xi_d)^{d-2}$, we see that our condition  on directions $\xi\in S^{d-1}\setminus (\alpha\cup\beta)$ for which  $n_{\alpha_{\xi}}\cdot n_{\beta_{\xi}}=0$  can be re-written as
$$
f(\xi):=\sum_{l=1}^{d-1} P_l^{\alpha}(\xi) P_l^{\beta}(\xi) = 0.
$$
In other words, the desired set is  $D_f=Z_f\cap(S^{d-1}\setminus (\alpha\cup \beta))$, where $Z_f$ is the set of zeros of $f$ on $S^{d-1}$.
By geometric considerations, $f$ is not identically zero (it is enough to look at the directions $\xi$ that are close to the ones that are parallel to $span$ $(\alpha, \beta)$).
If the closure of $D_f$ has a non-empty interior then $Z_f$ has a non-empty interior, i.e., there exists a spherical cap ${\mathcal B}_{\epsilon}(w)\subset Z_f$ of radius $\epsilon>0$ centered at $w\in S^{d-1}$. But this is impossible, for
replacing $\xi_d$ with $\sqrt{1-\xi_1^2-\dots-\xi_{d-1}^2}$ in the analytic expression for $f$, we obtain
 an analytic
function of variables $\xi_1,\dots,\xi_{d-1}$ in an open disc ${\mathcal B}_{\epsilon}(w)|w^{\perp}$ centered at the origin. This contradicts the fact that the set of zeros of an analytic function of several real variables is of Lebesgue measure zero (cf. \cite{O}).

Thus, the closure of $D_f$ has an empty interior, which means that $D_f$ is nowhere dense. The Lemma is proved.
\end{proof}

The following elementary result  will be crucial in the Proof of Theorem \ref{th2}.

\begin{lemma}\label{lines}
Let $U$ be an open subset of $S^{d-1}$ and let $\{l_i\}_{i=1}^4$ be four lines  in $\mathbb{E}^d, d \geq 3$, not passing through the origin, such that for any $\xi \in U$ the subspace $\xi^{\perp}$ intersects each line  at a single point $v_i(\xi)$. Assume also that $|v_1(\xi)v_2(\xi)| = |v_3(\xi)v_4(\xi)|$ for any $\xi \in U$.

1) If $l_1$ and $l_2$ are parallel, then all four lines are parallel. In addition, there exists a translation $b$, such that $l_3 = l_1 + b$ and $l_4 =l_2 + b$ or $l_3 = l_2 + b$ and $l_4 = l_1+b$ (the last relation can be re-written as $l_3 = -l_1+c, l_4 = -l_2+c$ for $c \in \mathbb{E}^d$; here $-l$ is the reflection of the  line $l$ in   the origin, i.e., $-l=\{-x:\,x\in l\}$).

2) If $l_1$ and $l_2$ are not parallel, then one of the following holds: $l_1 =\pm l_3, l_2 = \pm l_4$ or $l_1 =\pm  l_4, l_2 =\pm l_3$.

3) If $l_1$ and $l_2$ are not parallel and $\textrm{dim}(\textrm{span}(l_1\cup  l_2))=3$, then one of the following holds: $l_1 = l_3, l_2 = l_4$, or $l_1 = -l_3, l_2 = -l_4$, or $l_1 = l_4, l_2 = l_3$, or $l_1 = -l_4, l_2 = -l_3$.

\end{lemma}
\begin{proof}
We start the proof with some elementary algebraic observations.

Parameterize each line $l_i(t) = b_i + ta_i$, such that $t \in \mathbb{R}$, $|a_i|=1$ and $b_i \cdot a_i =0$. Notice, that the choice of the directional vectors $a_i$ is determined up to a sign, and the value of the parameter $t$ for the points of intersection can be found from the condition $\xi \cdot(b_i + t a_i)=0$, i.e. $v_i(\xi)= b_i - \frac{\xi \cdot b_i}{\xi \cdot a_i} a_i$. Hence, the condition of the lemma can be re-written as
\begin{equation}\label{main}
\left|b_1 - \frac{\xi \cdot b_1}{\xi \cdot a_1} a_1-b_2 + \frac{\xi \cdot b_2}{\xi \cdot a_2} a_2\right|=\left|b_3 - \frac{\xi \cdot b_3}{\xi \cdot a_3} a_3-b_4 + \frac{\xi \cdot b_4}{\xi \cdot a_4} a_4\right|, \quad \forall \xi \in U.
\end{equation}
In other words,
\begin{equation}\label{fraction}
\frac{P_1(\xi)}{Q_1(\xi)}= \frac{P_2(\xi)}{Q_2(\xi)}, \quad \forall \xi \in U,
\end{equation}
where
$$
P_1(\xi) = |(\xi \cdot a_1) (\xi \cdot a_2)(b_1-b_2) - (\xi \cdot b_1) (\xi \cdot a_2) a_1 + (\xi \cdot a_1) (\xi \cdot b_2) a_2|^2,
$$
$$
P_2(\xi) = |(\xi \cdot a_3) (\xi \cdot a_4)(b_3-b_4) - (\xi \cdot b_3) (\xi \cdot a_4) a_3 + (\xi \cdot a_3) (\xi \cdot b_4) a_4|^2,
$$
$$
Q_1(\xi) = |(\xi \cdot a_1) (\xi \cdot a_2)|^2,\qquad Q_2(\xi) = |(\xi \cdot a_3) (\xi \cdot a_4)|^2.
$$
By the direct computations, we have
$$
P_1(\xi) = (\xi \cdot a_1)^2(\xi \cdot a_2)^2|b_1-b_2|^2+(\xi \cdot b_1)^2(\xi \cdot a_2)^2+(\xi \cdot a_1)^2(\xi \cdot b_2)^2+
$$
$$+2(\xi \cdot b_1)(\xi \cdot a_2)^2(\xi \cdot a_1) (a_1 \cdot b_2)-2(\xi \cdot a_1)(\xi \cdot b_2)(\xi \cdot a_2)(\xi \cdot b_1) (a_1 \cdot a_2)+
$$
$$
+2(\xi \cdot a_1)^2(\xi \cdot a_2) (\xi \cdot b_2) (b_1 \cdot a_2).
$$

To proceed, we show at first that the left-hand side of (\ref{fraction}) is reducible if and only if $l_1 \parallel l_2$.

Indeed, if $l_1 \parallel l_2$, then $a_1 = \pm a_2$, and the fractions are reducible.
 Using the fact that $a_1 \cdot b_2 = a_2 \cdot b_1 = 0$, we can re-write the left-hand side of (\ref{fraction}) as:
\begin{equation}\label{fraction2}
\frac{P_1(\xi)}{Q_1(\xi)} = \frac{(\xi \cdot a_1)^2|b_1-b_2|^2+\left(\xi \cdot(b_1-b_2)\right)^2}{(\xi \cdot a_1)^2}.
\end{equation}
The fraction in (\ref{fraction2}) is not reducible, for, otherwise, $a_1 = \pm (b_1-b_2)$, which is not possible due to $(b_1-b_2) \perp a_1$.

Now assume that the left-hand side of the equation (\ref{fraction}) is reducible by, say, $(\xi \cdot a_1)$. Then the second term of $P_1(\xi)$ is reducible by $(\xi \cdot a_1)$. Since $b_1 \perp a_1$, we obtain that $a_2 = \pm a_1$. Similarly, if the left-hand side of (\ref{fraction}) is reducible by $(\xi \cdot a_2)$, then the third term is reducible by $(\xi \cdot a_2)$, which gives $a_2=\pm a_1$ and $l_1 \parallel l_2$.

If the left-hand side is of (\ref{fraction}) is reducible, i.e. looks as (\ref{fraction2}), then its right-hand side must be reducible as well. To see this,  compare all directions $\xi \in S^{d-1}$ for which either of the fractions is not defined. In particular, for all $\xi \in S^{d-1}$ such that $\xi \cdot a_1=0$, we must have have $\xi \cdot a_3 = 0$ or $\xi \cdot a_4 =0$.
Without loss of generality, assume that $a_1 = \pm a_3$. Then,
$$
 (\xi \cdot a_1)^2|b_1-b_2|^2+\left(\xi \cdot(b_1-b_2)\right)^2 = \frac{P_2(\xi)}{(\xi \cdot a_4)^2}.
$$
Since the polynomial on the left-hand side of the previous equality is defined for any $\xi \in S^{d-1}$, the right-hand side must be reducible.

Thus, we obtain that $a_1=\pm a_2$ implies $a_3=\pm a_4$, i.e. $l_1 \parallel l_2$ implies $l_3 \parallel l_4$.

If both of the fractions are reducible  (i.e. $a_1=\pm a_2$, $a_3=\pm a_4$), the equality of denominators implies $a_1=\pm a_2 = \pm a_3 = \pm a_4$. Using the equality of the numerators, we have
\begin{comment}and, due to the fact that $a_i\cdot b_i=0$, the fractions above are not reducible.
Hence,  $P_1(\xi)=P_2(\xi)$ and $Q_1(\xi)=Q_2(\xi)$  for all $\xi\in U$. Since all polynomials are homogeneous of degree $4$, we can extend $P_1$, $P_2$, $Q_1$, $Q_2$ to the whole space ${\mathbb E^d}$ and, by analiticity, we have $P_1(\xi)=P_2(\xi)$ and $Q_1(\xi)=Q_2(\xi)$
for all $\xi\in {\mathbb E^d}$.

We prove 1).

Assume  that $l_1$ and $l_2$ are parallel, i.e., that $a_1=\pm a_2$. Then $Q_1=Q_2$ reads $(\xi\cdot a_1)^4=(\xi \cdot a_3)^2 (\xi \cdot a_4)^2$. The left-hand side of this identity is divisible
by $(\xi\cdot a_1)^3$, hence the right-hand side must be divisible as well. This gives $a_1=\pm a_3=\pm a_4$, i.e., all four lines are parallel. Now we will use $P_1=P_2$. It reads as
$$
|(\xi \cdot a_1)^2(b_1-b_2) - (\xi \cdot (b_1-b_2)) (\xi \cdot a_1) a_1 |^2=|(\xi \cdot a_1)^2(b_3-b_4) - (\xi \cdot (b_3-b_4)) (\xi \cdot a_1) a_1 |^2,
$$
or, using the fact that $a_1\cdot (b_1-b_2-(b_3-b_4))= 0$,
\end{comment}
\begin{equation}\label{algebra}
(\xi \cdot a_1)^4(|b_1-b_2|^2 -|b_3-b_4|^2)+ ((\xi \cdot (b_1-b_2))^2-(\xi \cdot (b_3-b_4))^2) (\xi \cdot a_1)^2=0.
\end{equation}

By taking $\xi = a_1$ in (\ref{algebra}), we obtain $|b_1-b_2|=  |b_3-b_4|$.
\begin{comment}
If  $|b_1-b_2|\neq  |b_3-b_4|$, dividing all terms in the above expression by
 $(\xi \cdot a_1)^4$, we see that $(\xi \cdot (b_1-b_2))^2-(\xi \cdot (b_3-b_4))^2$
must be divisible by $(\xi \cdot a_1)^2$. In other words,
$$
b_1-b_2-(b_3-b_4)=\pm \lambda a_1,\qquad \textrm{or}\qquad b_1-b_2+b_3-b_4=\pm \mu a_1,
$$
for some $\lambda, \mu \in {\mathbb R}$. Since all $b_i\in a_1^{\perp}$, we have $\lambda=0=\mu = 0$,
and the claim is proven.
\end{comment}
We see now that (\ref{algebra}) has reduced to
$$
((\xi \cdot (b_1-b_2))^2-(\xi \cdot (b_3-b_4))^2) (\xi \cdot a_1)^2=0,
$$
or to
$$
(\xi \cdot (b_1-b_2-(b_3-b_4)))(\xi \cdot (b_1-b_2+b_3-b_4))=0\qquad\forall\xi\in{\mathbb E^d}.
$$

This implies
$$
b_1-b_2 = b_3-b_4 \quad \textrm{or} \quad b_1-b_2 = b_4-b_3.
$$
The first part of the lemma is proved.

We prove 2).

In this case, as we have already seen, the fractions in equation (\ref{fraction}) are not reducible. We will consider the first four possibilities, other four options are obtained by changing indices $3\leftrightarrow 4$.
Comparing the denominators in (\ref{fraction}), we see that $a_1=\pm a_3$, $a_2=\pm a_4$.
Writing out  the numerators in (\ref{fraction}), we have
$$
(\xi\cdot a_1)^2(\xi\cdot a_2)^2(|b_1-b_2|^2-|b_3-b_4|^2)+(\xi\cdot a_2)^2((\xi\cdot b_1)^2-(\xi\cdot b_3)^2)+
$$
$$
+\,(\xi\cdot a_1)^2((\xi\cdot b_2)^2-(\xi\cdot b_4)^2)-
$$
$$
-\,2(\xi\cdot a_1)(\xi\cdot a_2)^2((a_1\cdot (b_1-b_2))(\xi\cdot b_1)-(a_1\cdot (b_3-b_4))(\xi\cdot b_3))+
$$
$$
+\,2(\xi\cdot a_1)^2(\xi\cdot a_2)((a_2\cdot (b_1-b_2))(\xi\cdot b_2)-(a_2\cdot (b_3-b_4))(\xi\cdot b_4))-
$$
$$
-\,2(\xi\cdot a_1)(\xi\cdot a_2)(a_1\cdot a_2)((\xi\cdot b_1)(\xi\cdot b_2)-(\xi\cdot b_3)(\xi\cdot b_4))=0.
$$
Since the second term must be divisible by $\xi\cdot a_1$ and the third one must be divisible by $\xi\cdot a_2$, we see that
$b_1-b_3=\lambda_1 a_1$ or $b_1+b_3=\mu_1 a_1$ for some $\lambda_1$, $\mu_1\in{\mathbb R}$ (observe that both conditions may not hold simulteneously, for, otherwise, summing them up, we would obtain
that $a_1$ is parallel to $b_1$) and $b_2-b_4=\lambda_2 a_2$ or $b_2+b_4=\mu_2 a_2$ for some $\lambda_2$, $\mu_2\in{\mathbb R}$ (again, both conditions may not hold
simulteneously, otherwise $a_2$ would be  parallel to $b_2$). If $b_1-b_3=\lambda_1 a_1$, then $l_1=l_3$, for, the parametric equation of $l_3$ becomes $l_3(t)=b_1-\lambda_1 a_1+t a_1$, $t\in{\mathbb R}$, which is the same as the one of $l_1(s)=b_1+a_1 s$, $s\in{\mathbb R}$, after taking $t=\lambda_1+s$.
Arguing similarly, we see that
 $l_1=\pm l_3$, $l_2=\pm l_4$. This finishes the proof of 2).

In order to prove 3), it remains   to exclude two cases $l_1= l_3$, $l_2=- l_4$, and $l_1=- l_3$, $l_2= l_4$, provided $\textrm{dim}(\textrm{span}(l_1, l_2))= 3$ (the cases obtained by changing the indices $3\leftrightarrow 4$ are excluded similarly).
We will consider the first case and the exclusion of the second case is similar. To this end, assume that $a_3=a_1$, $b_3=b_1$, and $a_2=a_4$, $b_2=-b_4$.
Then,  condition (\ref{main}) reads as
$$
\left|\left(b_1 - \frac{\xi \cdot b_1}{\xi \cdot a_1} a_1\right)-\left(b_2 - \frac{\xi \cdot b_2}{\xi \cdot a_2} a_2\right)\right|=\left|\left(b_1 - \frac{\xi \cdot b_1}{\xi \cdot a_1} a_1\right)+\left(b_2 - \frac{\xi \cdot b_2}{\xi \cdot a_2} a_2\right)\right|, \, \forall \xi \in U.
$$
It is satisfied only, provided
\begin{equation}\label{oshibka}
\Big(b_1 - \frac{\xi \cdot b_1}{\xi \cdot a_1} a_1\Big)\cdot \Big(b_2 - \frac{\xi \cdot b_2}{\xi \cdot a_2} a_2\Big)=0 \quad \forall \xi \in U.
\end{equation}
(Observe that (\ref{oshibka}) holds in the case $d\ge 4$, $\textrm{dim}(\textrm{span}(l_1, l_2))= 4$,
 and $span(a_1, b_1)\perp span(a_2,b_2)$).

We show that (\ref{oshibka})  leads to a contradiction in the case $\textrm{dim}(\textrm{span}(l_1, l_2))= 3$. We may assume that $d=3$.
We observe that indeed $\textrm{span}(a_1,b_1)=\textrm{span}(a_2,b_2)$. Passing to the common denominator in (\ref{oshibka}) and using analiticity, we have
$$
(b_1\cdot b_2)(\xi\cdot a_1)(\xi\cdot a_2)-(b_1\cdot a_2)(\xi\cdot a_1)(\xi\cdot b_2)-
$$
$$
(a_1\cdot b_2)(\xi\cdot b_1)(\xi\cdot a_2)+(a_1\cdot a_2)(\xi\cdot b_1)(\xi\cdot b_2)=0\qquad\forall\xi\in {\mathbb E^3}.
$$
Since $d=3$ and, by the assumption, $a_i \cdot b_i = 0$one of the values $b_1\cdot b_2$, $b_1\cdot a_2$, $a_1\cdot b_2$, $a_1\cdot a_2$ is not zero. Hence, the left-hand side of the previous equation is not identically zero. Without loss of generality,
$b_1\cdot b_2\neq 0$. Then, the expression
$$
-(a_1\cdot b_2)(\xi\cdot b_1)(\xi\cdot a_2)+(a_1\cdot a_2)(\xi\cdot b_1)(\xi\cdot b_2)
$$
is divisible by $\xi\cdot a_1$, and, as a consequence,
$$
\xi\cdot((a_1\cdot a_2)b_2-(a_1\cdot b_2)a_2)
$$
is divisible by $\xi\cdot a_1$. This gives
\begin{equation}\label{raz}
(a_1\cdot a_2)b_2-(a_1\cdot b_2)a_2=\mu a_1
\end{equation}
for some $\mu\in{\mathbb R}$.

Observe that $\mu \neq 0$. Indeed, if $\mu=0$, we have $a_1 \cdot a_2 =0, a_1 \cdot b_2 =0$, since $a_2 \perp b_2$. Then (\ref{oshibka}) implies
$$
(b_1\cdot b_2)(\xi\cdot a_1)(\xi\cdot a_2)-(b_1\cdot a_2)(\xi\cdot a_1)(\xi\cdot b_2)=0, \quad \forall \xi \in \mathbb{E}^3,
$$
or, equivalently
$$
(b_1\cdot b_2)(\xi\cdot a_2)=(b_1\cdot a_2)(\xi\cdot b_2), \quad \forall \xi \in \mathbb{E}^3.
$$
The last condition is not possible, due to $a_2 \perp b_2$. Thus, $\mu \neq 0$.

Similarly, since
$$
(b_1\cdot b_2)(\xi\cdot a_1)(\xi\cdot a_2)-(a_1\cdot b_2)(\xi\cdot b_1)(\xi\cdot a_2)
$$
is divisible by $\xi\cdot a_2$, the expression
$$
(\xi\cdot b_2)((a_1\cdot a_2)(\xi\cdot b_1)-(b_1\cdot a_2)(\xi\cdot a_1))
$$
is divisible by $\xi\cdot a_2$ as well, and we have
\begin{equation}\label{dva}
(a_1\cdot a_2)b_1-(b_1\cdot a_2) a_1=\lambda a_2,
\end{equation}
for some $\lambda\in{\mathbb R}$. Following an argument which is similar to the one above, it can be shown that $\lambda \neq 0$. Notice that, by (\ref{raz}), $a_1 \in \textrm{span}(b_2,a_2)$; and, by (\ref{dva}), $b_1 \in \textrm{span}(a_1,a_2)$, which implies that $\textrm{span}(a_1,b_1)=\textrm{span}(a_2,b_2)$, due to the fact that $a_1$ is not parallel to $a_2$.

 Finally, denote
$$
A(\xi)=(\xi \cdot a_1)b_1 - (\xi \cdot b_1) a_1,\qquad B(\xi)=(\xi \cdot a_2) b_2 - (\xi \cdot b_2)a_2.
$$
We know that for every $\xi\in {\mathbb E^3}$ these vectors belong to the same two-dimensional subspace of ${\mathbb E^3}$.
Using (\ref{oshibka}) and the analyticity of $A(\xi)$ and $B(\xi)$, we see that   $A(\xi)\cdot\xi=0$ and $B(\xi)\cdot\xi=0$ for all $\xi\in {\mathbb E^3}$, i.e.,
$A(\xi)$ and $B(\xi)$ belong to $\xi^{\perp}$.
Moreover, due to (\ref{oshibka}) they are orthogonal to each other.
 By taking $\xi\in \textrm{span}(a_1,b_1)$, we obtain two vectors that are orthogonal and parallel to each other at the same time. Hence, at least one of them is zero, and we obtain a contradiction.
The lemma is proved.
\end{proof}
\begin{lemma}\label{perpl}
Let $l_p, l_q, l_r$ be three distinct lines in $\mathbb{E}^d$ and $\xi \in S^{d-1}$, such that $\xi$ is not orthogonal to $l_i$ for $i=p,q,r$. Denote $v_i(\xi) = l_i \cap \xi^{\perp}$. Then the set of directions $\xi \in S^{d-1}$, such that non-zero vectors $v_q(\xi) - v_p(\xi)$ and $v_r(\xi) - v_p(\xi)$ are orthogonal, is a nowhere dense subset $Y_{pqr} \subset S^{d-1}$.
\end{lemma}

\begin{proof} Let $l_i(t) = b_i + ta_i$ $ (t \in \mathbb{R}; i=p,q,r)$, then $v_i(\xi) = b_i - \frac{b_i \cdot \xi}{a_i \cdot \xi}a_i$.
Consider a function $f: S^{d-1} \to \mathbb{R}$, defined as following

$$
f(\xi)=\left(b_q  - \frac{b_q \cdot \xi}{a_q \cdot \xi}a_q - b_p + \frac{b_p \cdot \xi}{a_p \cdot \xi}a_p \right) \cdot \left(b_r  - \frac{b_r \cdot \xi}{a_r \cdot \xi}a_r - b_p + \frac{b_p \cdot \xi}{a_p \cdot \xi}a_p \right).
$$
The condition of orthogonality of the given vectors in terms of $\xi$ is $f(\xi) = 0$. By geometrical considerations, $f \not \equiv 0$ on $S^{d-1}$. 
Since $a_i \cdot \xi \neq 0$, then, multiplying both sides by $(a_p \cdot \xi)(a_q \cdot \xi)(a_r \cdot \xi)$, the condition is equivalent to a polynomial equation of third degree in coordinates of $\xi$.

Using a similar argument as in Lemma \ref{perp}, we can show that the set of zeroes $Z_f \subset S^{d-1}$ of $f$ is nowhere dense on $S^{d-1}$.

\end{proof}

We will use the well-known Minkowski theorem (see, for example,  \cite{K}, Theorem 9, p. 282).

\begin{theorem}\label{Minkowski}
Suppose that $K$ and $L$ are polytopes in $\mathbb{E}^d, d \geq 2$, and suppose also that the facet unit normals $n^1,\ldots,n^l$ and corresponding facet areas $c_1,\ldots,c_l$ coincide, then $K$ coincides with $L$ up to a translation.
\end{theorem}

Our last  auxiliary result in this section is

\begin{theorem}[see \cite{M}, Theorem 1.3] \label{hedgehog}
Let $2 \leq j \leq d-1$ and let $f$ and $g$ be two continuous real-valued functions on $S^{d-1}$. Assume that  for any $j$-dimensional subspace  $\alpha$ and some vector $a_{\alpha}\in \alpha$, the restrictions of $f$ and $g$ onto $S^{d-1}\cap \alpha$ satisfy $f(-u)+a_{\alpha} \cdot u = g(u)$ $\quad \forall u \in \alpha \cap S^{d-1}$ or
$ f(u) + a_{\alpha} \cdot u = g(u)$ $\quad \forall u \in \alpha \cap S^{d-1}$.
Then there exists $b \in \mathbb{E}^d$ such that   $g(u) = f(u) + b \cdot u$ $\forall u\in S^{d-1}$ or $g(u) = f(-u) + b \cdot u$ $\forall u\in S^{d-1}$.
\end{theorem}

We remark also that if all $a_{\alpha}$ are zero, then $b=0$ (see also Lemma 1 from \cite{R2}, page 3431).

\section{Proof of Theorem \ref{th1}}

 We start with the case $2\le k= d-1$.

\subsection{Main idea}
We will   show that for all directions $\xi\in S^{d-1}$, the  projections of polytopes onto $\xi^{\perp}$ coincide up to a translation and  a reflection in the origin.
This will be achieved in three steps that can be briefly sketched as follows.

{\it Step 1}. Let $Y\subset S^{d-1}$ be a closed set of directions $\xi$ that are parallel to facets of $P$ or $Q$, and let $U_1$ be any connected component of $U=S^{d-1}\setminus Y$.
We will prove that, given any  two vectors $\xi_1$ and $\xi_2$ in  $U_1$, we have $\partial_{\xi_1}P=\partial_{\xi_2}P$ and $\partial_{\xi_1}Q=\partial_{\xi_2}Q$; see Lemma \ref{shadow}.

{\it Step 2}. We will prove that  the edges of the shadow boundaries $\partial_{\xi}P$ and $\partial_{\xi}Q$, $\xi \in U$, are in  a bijective correspondence.
We will apply Lemma \ref{length} to all corresponding edges of $\partial_{\xi}P$ and $\partial_{\xi}Q$ to conclude that they are parallel and have equal length.

{\it Step 3}. We will  show that the corresponding facets of  projections $P_{\xi}$ and $Q_{\xi}=T_{\xi}(P_{\xi})$, $\xi\in U$, are pairwise parallel.
We will apply Minkowski's Theorem  to conclude that  $P_{\xi}$ and $Q_{\xi}$, $\xi\in U$,   coincide up to a translation and a  reflection in the origin. Then, we will use the density argument to conclude that the last statement holds for all directions $\xi\in S^{d-1}$.

Finally, we  will apply Theorem \ref{hedgehog} with $f$ and $g$ being  the support functions of polytopes.

\subsection{Proof}
{\it Step 1}.
Fix any facet $F$ of $P$ or $Q$ with an outer unit normal $\eta$. Any direction $\xi \in S^{d-1}$ which is parallel to $F$  belongs to $ \eta^{\perp}$. Since polytopes have a finite number of facets, the set $Y$ of all such directions $\xi$ is a union of a finite number of great subspheres of $S^{d-1}$.  Denote  $U=S^{d-1} \setminus Y $. Since $Y$ is closed, $U$ is open.

Fix any $\xi \in U$. Any vertex $v_{\xi} \in P_{\xi}$ is a projection of some vertex (not an edge), otherwise $\xi$ is parallel to some facet. The same holds for edges of $P_{\xi}$, i.e. a pre-image of any edge of $P_{\xi}$ is an edge of $P$, otherwise $\xi$ is parallel to some $2$-dimensional face of $P$.

Let $U_1$ be any non-empty  linearly connected component of $U$.  Observe that it  is contained in an open hemisphere $S_{\zeta}^{d-1}$ for some $\zeta\in S^{d-1}$.
Note also that $U_1$ is geodesically convex, i.e., for any two directions $\xi_1$ and $\xi_2$ in $U_1$ there exists an arc $[\xi_1,\xi_2]$ of a great circle of $S^{d-1}$ such that $[\xi_1,\xi_2]\subset U_1$.
Indeed, since the boundary $\partial U_1$ of $U_1$ is a finite union of closed pieces of great subspheres of $S^{d-1}$, $\partial U_1=\bigcup\limits_{k=1}^jS_k$, we see that $U_1 = S^{d-1} \cap {\mathcal H}$, where ${\mathcal H}$ is a finite convex intersection of $j$ half-spaces defined by hyperplanes, passing through the origin and  containing $S_k$. Since ${\mathcal H}$ is convex, the interval $\xi_1\xi_2$ connecting $\xi_1$ and $\xi_2$ belongs to ${\mathcal H}$. The projection of this interval onto $S^{d-1}$ is an arc $[\xi_1,\xi_2]$ of the unit circle, which is contained in $U_1$.

We have

\begin{comment}
\begin{lemma}\label{shadow}
For any $\xi_0 \in U_1$ there exists a closed neighborhood $V_{\xi_0} \subset U_1$, such that $\forall \xi \in V_{\xi_0}$ we have $\partial_{\xi} P = \partial _{\xi_0} P$.
\end{lemma}

\begin{proof}
Consider all edges $\{a^1, a^2,...,a^l\} \subset P$ that are connected to all vertices $\{v^1,v^2,...,v^k\}$. Assume $u^i$ is a unit directional vector of $a^i, i=1,2,..,l$. Then we have $(\xi_0 \cdot u^i)^2 \neq 1$, since $\xi_0 \not \parallel u^i$, for any $i=1,2,..l$. The scalar product is a continuous function, thus for a fixed $i$ there exists an open set $V^i \subset U_1$, such that for any $\xi \in V_i$ we have $(\xi \cdot u^i )^2 \neq 1$. Observe that $\xi_0 \in V_i$ for all $i=1,2,..l$, hence $\xi_0 \in \cap_{i=1}^lV^i=V_0$. Finite intersection of a union of open sets $V^i$ is open, hence $V_0$ is open. Take a proper closed subset $V_{\xi_0} \not \subseteq V_0$, such that $\xi_0 \in V_{\xi_0}$.
\end{proof}
\end{comment}

\begin{lemma}\label{shadow}
For any $\xi_1, \xi_2 \in U_1$ the shadow boundaries coincide, $\partial_{\xi_1}P = \partial_{\xi_2}P$.
\end{lemma}

\begin{figure}[h]
\includegraphics[scale=0.5]{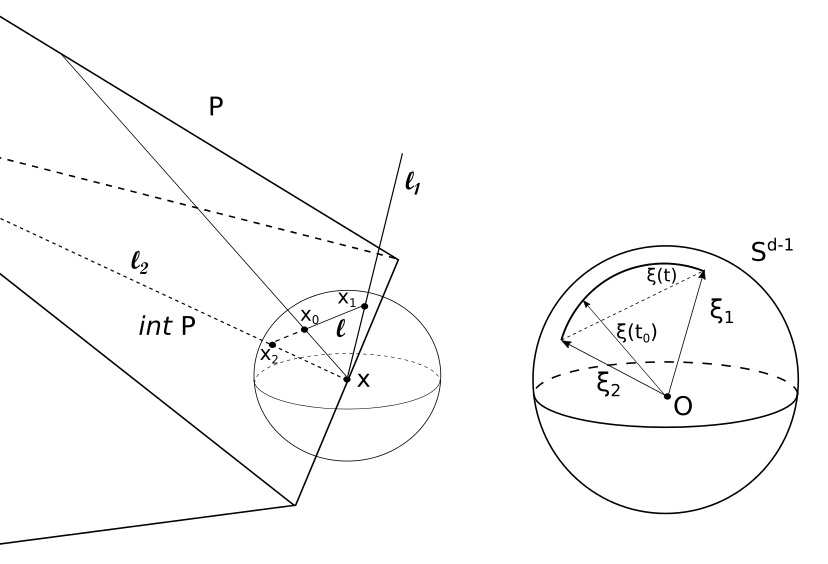}
\caption{Preserving of shadow boundaries.} \label{shad}
\end{figure}

\begin{proof}
Assume the opposite, there exist two distinct vectors $\xi_1, \xi_2 \in U_1$, such that $\partial_{\xi_1} P \neq \partial_{\xi_2} P$. Then there exists $x \in \partial_{\xi_1}P$, but $x \not \in \partial_{\xi_2}P$, and  we may construct two  half-lines $l_1(t_1) = x + t_1 \xi_1, t_1 \geq 0,$ and $l_2(t_2) = x + t_2 \xi_2, t_2 \geq 0$ (or $t_2 \leq 0$, if necessary) to obtain
$$
l_1 \cap int P = \emptyset \quad \textrm{and} \quad l_2 \cap int P \neq \emptyset.
$$
Take a small enough $\varepsilon > 0$, such that the ball $B(x,\varepsilon)$ intersects only faces of $P$ that contain $x$ (see Figure \ref{shad}). Choose $x_2 \in l_2 \cap int P \cap B(x,\varepsilon)$ and $x_1 \in l_1 \cap B(x,\varepsilon), x_1 \neq x$. Notice that $x_1 \not \in \partial P$, otherwise $\frac{x_1-x}{|x_1-x|} \not \in U$, which is not possible since $\xi_1$ is parallel to $ \frac{x_1-x}{|x_1-x|}$.

We project
 the interval $x_1x_2 =\{ tx_1 + (1-t) x_2, t \in [0,1]\}$ onto $S^{d-1}$ and obtain an arc $\xi(t) = \frac{tx_1 + (1-t) x_2 - x}{|tx_1 + (1-t) x_2 - x|}$ of the unit circle $span \{\xi_1,\xi_2\} \cap S^{d-1}$.
On the other hand, there exists a point $x_0$, such that $x_0=t_0x_1+(1-t_0)x_2 \in \alpha$ for some $t_0 \in (0,1)$, where $\alpha \subset \partial P$ is a facet of $P$ containing $x_0$. We have $\xi(t_0) \not \in U_1$, which  contradicts  the fact that  $U_1$ is geodesically convex.
\end{proof}

{\it Step 2}.
Our next goal is to show  that the edges of shadow boundaries $\partial_{\xi}P$ and $\partial_{\xi}Q$, $\xi\in U_1$, are in a  bijective correspondence. Moreover, we will prove that the corresponding edges are parallel and have equal length.

Assume that for a fixed $\xi \in U_1$ the projection $P_{\xi}$ has $k$ vertices $\{(v_1)_{\xi}, (v_2)_{\xi}, \dots,$ $ (v_k)_{\xi}\}$ which  are the projections of $k$ vertices $\{v_1,v_2,\ldots,v_k\}$ of $P$. In each hyperplane $\xi^{\perp}$ we consider a rigid motion $T_{\xi}$, such that $T_{\xi}(P_{\xi}) = Q_{\xi}$ (if there are several such $T_{\xi}$, take any). It is clear that every vertex $(v_i)_{\xi} \in P_{\xi}$ is mapped onto some vertex  $(\tilde{v}_j)_{\xi} =T_{\xi}((v_i)_{\xi})\in Q_{\xi}$ and every  edge of $P_{\xi}$ is mapped onto some edge of $Q_{\xi}$ (notice also  that the pre-image of any vertex of $ Q_{\xi}$ is a vertex of $Q$  and the pre-image of any edge of $Q_{\xi}$ is an edge of $Q$).
This implies that for any  $\xi\in U_1$ we obtain a bijective correspondence $f_{\xi}$ between the set of all vertices  $\{v_1,v_2,...,v_k\} $ of the shadow boundary $\partial_{\xi}P$ and the set of all  vertices $\{\tilde{v}_1,\tilde{v}_2,...,\tilde{v}_k\}$ of the shadow boundary $\partial_{\xi}Q$.

\begin{figure}
  \centering
  \includegraphics[scale=0.5]{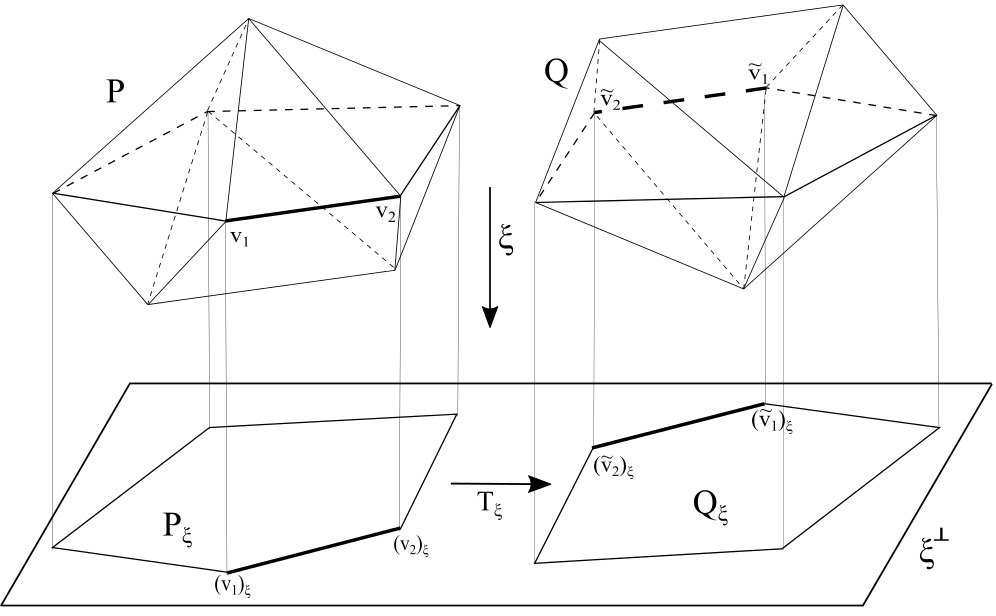}\\
  \caption{Correspondence of vertices through projections}\label{permute}
\end{figure}

Take a closed spherical cap with a  non-empty interior $W \subset U_1$. For any $\xi \in W$ we have at least one $T_{\xi}$ satisfying $T_{\xi}(P_{\xi}) = Q_{\xi}$. Hence, for any $\xi \in W$, we have at least one map $f_{\xi}: \{v_1,v_2,...,v_k\} \to \{\tilde{v}_1,\tilde{v}_2,...,\tilde{v}_k\} $, such that $f_{\xi}(v_i) = \tilde{v}_{\sigma_{\xi}(i)}$,  and $\sigma_{\xi}$ is a permutation of the set $\{1,2,...,k\}$ satisfying $(\tilde{v}_{\sigma_{\xi}(i)})_{\xi}=T_{\xi}((v_i)_{\xi})$ (see Figure \ref{permute}). The set  of all such possible maps $\{f_{\xi}\}_{\xi \in W}$ is finite. We have
$$
W = \bigcup\limits_{\sigma \in {\mathcal P}_k} V_{\sigma},\qquad V_{\sigma} = \{\xi \in W:\, \exists f_{\xi} \,\,\textrm{such that}\,\,f_{\xi}(v_i)=\tilde{v}_{\sigma(i)}\quad\forall i=1,\dots,k\},
$$
where $ {\mathcal P}_k$ is the set of all permutations of $\{1,2,...,k\}$.

Observe that each $V_{\sigma}$ is a closed set (it might be empty). Indeed, let $(\xi_k)_{k=1}^{\infty}$ be a convergent sequence of points of a non-empty $V_{\sigma}$, and let $\lim\limits_{k\to\infty}\xi_k=\xi$.
We have $T_{\xi_k}((v_i)_{\xi_k})=(\tilde{v}_{\sigma(i)})_{\xi_k}$, i.e.,
\begin{equation}\label{close}
T_{\xi_k}(v_i-(v_i\cdot\xi_k)\xi_k)=\tilde{v}_{\sigma(i)}-(\tilde{v}_{\sigma(i)}\cdot\xi_k)\xi_k,
\end{equation}
where $\sigma$ is independent of $\xi_k$.
For each $\xi\in W$ extend  the operator $T_{\xi}$  acting from  $\xi^{\perp}$ to $\xi^{\perp}$ to the operator
${\mathbb T}_{\xi}$ acting from
 ${\mathbb E^d}$ to ${\mathbb E^d}$ as
${\mathbb T}_{\xi}(a)=\Phi_{\xi}(a)+b_{\xi}$. Here,
$\Phi_{\xi}\in O(d)$ is defined as $\Phi_{\xi}|_{\xi^{\perp}}=\varphi_{\xi}$, $\Phi_{\xi}(\xi)=\xi$, where $T_{\xi}(a_{\xi})=\varphi_{\xi}(a_{\xi})+b_{\xi}$. Equation (\ref{close})
in terms of ${\mathbb T}_{\xi_k}$ can be re-written as
$$
{\mathbb T}_{\xi_k}(v_i-(v_i\cdot\xi_k)\xi_k)=\Phi_{\xi_k}(v_i)-(v_i\cdot\xi_k)\xi_k+b_{\xi_k}   =\tilde{v}_{\sigma(i)}-(\tilde{v}_{\sigma(i)}\cdot\xi_k)\xi_k.
$$

Without loss of generality, both polytopes $P$ and $Q$ are located inside a large ball. Hence, the entries $\{a_{ij}^k\}_{k \in \mathbb{N}}, i,j=1,2,...,d$, of the matrix corresponding to transformations $\Phi_{\xi_k}$ and the coordinates $b^k_j, j=1,\ldots,d$ of vector $b_{\xi_k}$ are bounded functions of $\xi_k$. By compactness, we can assume that all $\{a_{ij}^k\}_{k \in \mathbb{N}}$, and $\{b^k\}_{k \in \mathbb{N}}$ are convergent to the corresponding entries of $\tilde{\Phi}_{\xi} = \lim_{k \to \infty} \Phi_{\xi_k}$ and $\tilde{b}_{\xi} = \lim_{k \to \infty}b_{\xi_k}$ respectively. This yields
$$
\tilde{{\mathbb T}}_{\xi}\left((v_i)_{\xi}\right)=\tilde{\Phi}_{\xi}(v_i-(v_i\cdot\xi)\xi)+\tilde{b}_{\xi} =T_{\xi}((v_i)_{\xi})=\tilde{v}_{\sigma(i)}-(\tilde{v}_{\sigma(i)}\cdot\xi)\xi=(\tilde{v}_{\sigma(i)})_{\xi},
$$
where $\tilde{\mathbb{T}}_{\xi} = \lim_{k \to \infty} \mathbb{T}_{\xi_k}$. In other words, there exists $\tilde{T}_{\xi}$, such that the corresponding $\tilde{f}_{\xi}$ satisfies $\tilde{f}_{\xi}(v_i)=\tilde{v}_{\sigma(i)}, \forall i=1,\dots,k$. This means that $\xi\in V_{\sigma}$ and $V_{\sigma}$ is a closed.

By the Baire category Theorem (see, for example,  \cite{R}, pages 42-43) there exists a permutation $\sigma_o$ such that the interior $U_o$ of $V_{\sigma_o}$ is non-empty. (This could be also easily seen as follows.
Enumerate  ${\mathcal P}_k$, and take the first set $V_{\sigma_1}$. If its interior is empty, it is nowhere dense, and for every open spherical cap ${\mathcal B}_1$ in $W$
 there exists a smaller cap ${\mathcal B}_2$ that is free of points of
$V_{\sigma_1}$, ${\mathcal B}_2\cap V_{\sigma_1}=\emptyset$. Repeat the procedure with $V_{\sigma_2}$ and ${\mathcal B}_2$ instead of
$V_{\sigma_1}$ and ${\mathcal B}_1$. After finitely many steps, unless we meet some $V_{\sigma}$ with a non-empty interior, we will optain a spherical cap that  does not intersect $W$, which is impossible).

Observe that if $v_i$ and $v_j$ are connected by an edge $v_iv_j \subset P$, such that $(v_i)_{\xi}(v_j)_{\xi}$ is an edge of $P_{\xi}$, $\xi\in U_o$, then $\tilde{v}_{\sigma(i)}\tilde{v}_{\sigma(j)}$ is an edge of $Q$.
Now, we can apply Lemma \ref{length} to all such pairs of edges $v_iv_j \in \partial_{\xi}P$ and $\tilde{v}_{\sigma(i)}\tilde{v}_{\sigma(j)} \in \partial_{\xi}Q$ for any $\xi \in U_o \subset S^{d-1}$. We see that
these edges are parallel and have equal length.
Thus,
all corresponding edges  belonging to  the shadow boundaries $\partial_{\xi}P$ and $\partial_{\xi}Q$,  $\xi \in U_o$, are parallel and have equal lengths. Applying  Lemma \ref{shadow}, we  conclude that the last statements holds for  $U_1$ instead of $U_o$.

{\it Step 3}.
We will  show that  the projections of both polytopes in the directions of $U_1$ coincide up to a translation and a reflection in the origin.
To do this, we will  use Minkowski's Theorem about uniqueness  (up to a translation) of polytopes with parallel facets having the same volume. Our polytopes will be
$P_{\xi}$ and $Q_{\xi}$, $\xi\in U_1$.

 Fix any $\xi \in U_1$. Then $P_{\xi}$ and $Q_{\xi}$ are two $(d-1)$-dimensional polytopes in $\xi^{\perp}$, such that $T_{\xi} (P_{\xi}) = Q_{\xi}$.  Since the map $T_{\xi}$ is a bijection between the sets of all  facets of $P_{\xi}$ and $Q_{\xi}$,
 for any  facet $\alpha$ of $P_{\xi}$ there exists a unique  facet $\tilde{\alpha}$ of $ Q_{\xi}$ such that $T_{\xi} ( \alpha ) = \tilde{\alpha}$.

We will show at first that $\alpha$ is parallel to $\tilde{\alpha}$.
Consider the $(d-2)$-dimensional affine subspaces $\Pi$, $\tilde{\Pi}$ of $\xi^{\perp}$ such that $\alpha \subset \Pi$ and $\tilde{\alpha} \subset \tilde{\Pi}$.
We claim  that $\tilde{\Pi}$ is parallel to $ \Pi$. In other words,  the outer unit normals $n_{\alpha}$, $n_{\tilde{\alpha}} \in\xi^{\perp}$ of
 $\Pi$ and $\tilde{\Pi}$ satisfy  $n_{\tilde{\alpha}} = \pm n_{\alpha}$.

To prove the claim, observe that
there exist $d-2$ linearly independent non-zero directional vectors $\{(a_1)_{\xi},\ldots, (a_{d-2})_{\xi} \}$ of edges of $\alpha$ having a common vertex, such that $span \{ (a_1)_{\xi},\ldots, (a_{d-2})_{\xi} \}$ is parallel to $\Pi$ (here, a directional vector of an edge is any non-zero vector parallel to the edge). The directional vectors of edges $\{(a_1)_{\xi},\ldots, (a_{d-2})_{\xi} \}$ are the projections of directional vectors of the edges $\{a_1,\ldots,a_{d-2}\} \subset \partial_{\xi}P$. The directional vectors $\{\tilde{a}_1,\ldots, \tilde{a}_{d-2}\} \subset \partial_{\xi}Q$ of the corresponding edges are parallel, i.e. $a_i$ is parallel to $\tilde{a}_i$ for $i=1,\ldots,d-2$. The same holds true for their projections onto $\xi^{\perp}$, $(a_i)_{\xi} $ is parallel to $(\tilde{a}_i)_{\xi}$ for $i=1,\ldots,d-2$. We conclude that for the $(d-2)$-dimensional affine subspace $\tilde{\Pi}$ containing $\tilde{\alpha}$ we have $T_{\xi}(\Pi) = \tilde{\Pi}$ and $\tilde{\Pi}$ is parallel to $ span \{( \tilde{a}_1)_{\xi},\ldots,(\tilde{a}_{d-2})_{\xi} \} = span \{(a_1)_{\xi},\ldots,(a_{d-2})_{\xi}\}$. The claim is proved.

Since $T_{\xi}$ is an isometry, we have $\textrm{vol}_{d-2} (\alpha) = \textrm{vol}_{d-2} (\tilde{\alpha})$. Now assume that both polytopes $P_{\xi}$ and $Q_{\xi}$ have $l$ facets $\{\alpha_i\}_{i=1}^l \subset P_{\xi}$ and $\{\tilde{\alpha}_i\}_{i=1}^l \subset Q_{\xi}$ with corresponding outer normals $\{n_i\}_{i=1}^l$ and $\{\tilde{n}_i\}_{i=1}^l$.
We will show  that up to a nowehere dense subset of directions $\xi$  (defined below)
 the projections $P_{\xi}$ and $Q_{\xi}$ are translates of each other (up to a reflection in the origin). To be able to  apply
 Minkowski's Theorem to $P_{\xi}$ and $Q_{\xi}$  (or  $-P_{\xi}$ and $Q_{\xi}$) we have to show that
$n_i = \tilde{n}_i$ for all $i=1,\ldots,l$, or $-n_i = \tilde{n}_i$ for all $i=1,\ldots,l$.

Assume that the last statement is not true,  and let  $\alpha_i, \alpha_j \subset P_{\xi}$ and $\tilde{\alpha}_i, \tilde{\alpha}_j \subset Q_{\xi}$, be  two pairs of facets such that $\alpha_i \cap \alpha_j \neq \emptyset$, $T_{\xi}(\alpha_i) = \tilde{\alpha}_i$, $T_{\xi}(\alpha_j) = \tilde{\alpha}_j$, but $\tilde{n}_i = n_i, \tilde{n}_j = -n_j$, (since $T_{\xi}$ is an isometry,  $\tilde{\alpha}_i \cap \tilde{\alpha}_j \neq \emptyset$).
Consider  two cases,
 $n_i \cdot n_j \neq 0$ and  $n_i \cdot n_j = 0$. The first case $n_i \cdot n_j \neq 0$ is impossible, for
  $$
  n_i \cdot (-n_j) =\tilde{n}_i \cdot \tilde{n}_j = \varphi_{\xi}(n_i) \cdot \varphi_{\xi} (n_j) = \varphi^T_{\xi} \varphi_{\xi}(n_i) \cdot n_j = n_i \cdot n_j,
  $$
yields $n_i \cdot n_j = 0$.
Here $\varphi^T_{\xi}$ stands for the transpose operator of $\varphi_{\xi}$, and we used  the fact that $\varphi^T_{\xi} \varphi_{\xi}=I$ due to $\varphi_{\xi} \in O(d-1,\xi^{\perp})$.

 To exclude the case $n_i \cdot n_j = 0$, we recall that the pre-image of any facet of $P_{\xi}$ or $Q_{\xi}$ is an $(d-2)$-dimensional face of $P$ or $Q$ respectively. We  apply Lemma \ref{perp} to the pairs of subspaces that are parallel to all pairs of   $(d-2)$-dimensional faces of $P$ and $Q$. We obtain a closed nowhere dense subset  $Y_1\subset U$,  such that  the normals of the corresponding facets of $P_{\xi}$ and $Q_{\xi}$, $\xi\in Y_1$, are orthogonal.
Now we may repeat our considerations for the case $n_i \cdot n_j \neq 0$ on  $U_1\setminus Y_1$ instead of $U_1$. We obtain that for every $\xi\in U_1\setminus Y_1$  only one of the choices
$n_i = \tilde{n}_i$ for all $i=1,\ldots,l$, or $-n_i = \tilde{n}_i$ for all $i=1,\ldots,l$, holds.

We conclude that for all $\xi\in U_1\setminus Y_1$,  we have $n_i = \tilde{n}_i$ for all $i=1,\ldots,l$, or $-n_i = \tilde{n}_i$ for all $i=1,\ldots,l$.
Also, recall that $c_i = \tilde{c}_i$ for all $i=1,\ldots,l$, where $c_i=\textrm{vol}_{d-2}(\alpha_i)$, $\tilde{c}_i=\textrm{vol}_{d-2}(\tilde{\alpha}_i)$;
now we may apply Theorem \ref{Minkowski} to $P_{\xi}$ and $Q_{\xi}$ or to $-P_{\xi}$ and $Q_{\xi}$.

Thus, for every $\xi\in U_1\setminus Y_1$  there exists a vector $b_{\xi} \in \xi^{\perp}$, such that $Q_{\xi} = P_{\xi} + b_{\xi}$ or $Q_{\xi} = -P_{\xi} + b_{\xi}$.
We can repeat the above argument for every connected component of $U$.  We see that for every $\xi\in U\setminus Y_1$ the supporting functions $h_P$ and $h_Q$ satisfy
\begin{equation}\label{plus}
h_Q(u) = h_P(u)+b_{\xi} \cdot u \quad \textrm{for any} \quad u \in \xi^{\perp},
\end{equation}
or
\begin{equation}\label{minus}
h_Q(u) = h_P(-u)+b_{\xi} \cdot u \quad \textrm{for any} \quad u \in \xi^{\perp}.
\end{equation}
Here we used the fact that $h_{P_{\xi}}(u)=h_P(u)$ $\forall u\in\xi^{\perp}$, see (\cite{Ga}, (0.21), p.17).

It is not difficult to see that one of the above equalities (or both) hold for every $\xi\in S^{d-1}$.
Indeed, since $Y\cup Y_1$ is nowhere dense on $S^{d-1}$,
 for any open neighborhood $V_{\xi}\subset S^{d-1}$ of any $\xi \in Y\cup Y_1$ we have $V_{\xi} \cap U \neq \emptyset$. Hence, there exists a sequence $\{\xi_k\}_{k \in \mathbb{N}} \subset U$, such that $\lim_{k \to \infty} \xi_k = \xi$. By taking a convergent subsequence if necessary, we conclude that (the argument is very similar to the one in  the proof of Lemma 6, p.  3435, \cite{R2}) there exists a limit $ \lim\limits_{k \to \infty} b_{\xi_k}=b_{\xi}  \in \xi^{\perp}$ for $\xi \in Y\cup Y_1$, such that (\ref{plus}) or (and)  (\ref{minus}) holds.

To finish the proof in the case $k=d-1$, we apply Theorem \ref{hedgehog} with $j=k$, and $f=h_P$, $g=h_Q$.
We obtain that there exists $b\in {\mathbb E^d}$ such that
$h_Q(u) = h_P(u)+b \cdot u $ for all $ u \in{\mathbb E^d}$,
or
$h_Q(u) = h_P(-u)+b \cdot u$ for all $ u \in {\mathbb E^d}$. Using the well-known properties of the support functions (\cite{Ga}, pp. 16-18) we obtain the desired result in the case
$k=d-1$.

\subsection{Proof of Theorem \ref{th1} in the case $2\leq k < d-1$}
We use   induction on $k$.
Let $H$ be any $(k+1)$-dimensional subspace of ${\mathbb E^d}$. We apply the result for $d=k+1$ to the bodies $P|H$ and $Q|H$ and their projections
$(P|H)|J$ and $(Q|H)|J$ for all $k$-dimensional subspaces $J\subset H$.
We obtain that $P|H= Q|H + b_H$ or $P|H = -Q|H+b_H$ for all $H$. We proceed and the result follows after finitely many steps.
The proof of Theorem \ref{th1} is completed.

\section{Proof of Theorem \ref{th2} }

We start with the case $2\le k= d-1$.

\subsection{Main idea}
We will   show that for all directions $\xi\in S^{d-1}$, the  sections of polytopes by $\xi^{\perp}$ coincide up to a reflection in the origin.
This will be achieved in four steps, which can be briefly described as follows.

{\it Step 1}. Let $Y\subset S^{d-1}$ be a closed set of  directions $\xi$  for which $\xi^{\perp}$ contains any of the vertices of $P$ or $Q$, and let $U_1$ be any connected component of
$U=S^{d-1}\setminus Y$.
We will prove that given any  two directions $\xi_1$ and $\xi_2$ in  $U_1$, the subspaces $\xi_1^{\perp}$ and $\xi_2^{\perp}$ intersect the same set of edges of $P$ (and $Q$);
 see Lemma \ref{edges}.

Denote by  $E=E(U_1)$ and $\tilde{E}=\tilde{E}(U_1)$ the sets of edges of $P$ and $Q$
 that are intersected by $\xi^{\perp}$, $\xi\in U_1$. Denote also by
$L=L(U_1)$ and $\tilde{L}=\tilde{L}(U_1)$
    the sets of lines containing the intersected edges  from $E$ and $\tilde{E}$.

{\it Step 2}.
We will prove that there exists an open non-empty subset $U_o$ of $U_1$ such that
for all  directions $\xi\in U_o$  all rigid motions $T_{\xi}$ ``look alike". More precisely, for all $\xi\in U_o$ and for any $l_i\in L$ the vertices  $v_i(\xi)=\xi^{\perp}\cap l_i$ of sections $P\cap \xi^{\perp}$ are all mapped into the vertices of $Q\cap\xi^{\perp}$ of the form $\tilde{v}_{j(i)}(\xi)=\xi^{\perp}\cap \tilde{l}_{j(i)}$, where the line  $\tilde{l}_{j(i)}\in \tilde{L}$ is fixed and  $j=j(i)$   is independent of $\xi$.

{\it Step 3}.
 Using Lemma \ref{lines} we will show that the corresponding lines  in $L$ and $\tilde{L}$ coincide up to a reflection in the origin.

{\it Step 4}.  We  will apply Theorem \ref{hedgehog}    with $j=d-1$, $a_{\alpha}=0$, $b=0$, and $f$ and $g$ being  the radial functions of polytopes.

\subsection{Proof} {\it Step 1}.
 Consider any vertex $v$ of  $P$ or $Q$. If a  hyperplane $\xi^{\perp}$ contains $v$ then  $\xi \cdot v=0$. Hence, the set of all such directions $\xi$ is a sub-sphere $v^{\perp} \cap S^{d-1}$. Since both polytopes have a finite number of vertices, the union  $Y$ of all such sub-spheres  is closed, hence its complement $U = S^{d-1} \setminus Y$ is open in $S^{d-1}$.

Fix any connected component $U_1$ of $U$.
\begin{lemma}\label{edges}
For any two distinct vectors $\xi_1, \xi_2 \in U_1$, the hyperplanes $\xi_1^{\perp}$ and $\xi_2^{\perp}$ intersect the same set of edges of $P$.
\end{lemma}
\begin{figure}[h]
\includegraphics[scale=0.4]{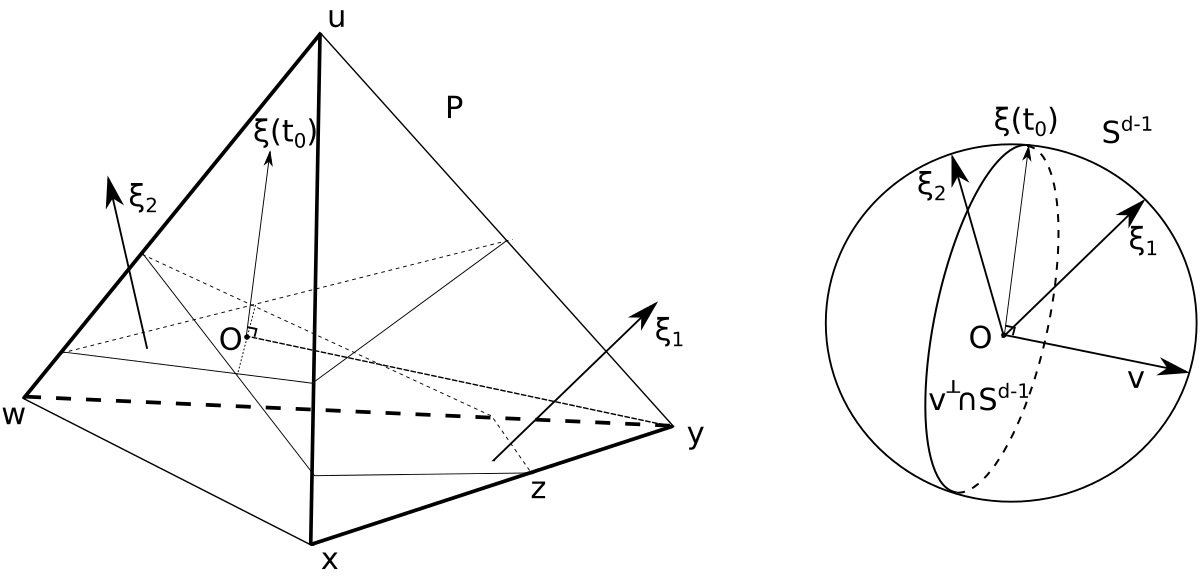}
\caption{The sets $A_1 = \{ xy,yw,wu,ux\}$, $ A_2 = \{ uw,ux,uy\}$.}\label{geoconv}
\end{figure}
\begin{proof}
Let $A_i$ be the set of all interiors of all edges of $P$ that have non-empty intersections with $\xi_i^{\perp}$, $i=1,2$. We claim that $A_1 = A_2$.

Assume that $A_1 \neq A_2$ (see Figure \ref{geoconv}). Then, there exists an edge $xy$ such that
$xy \cap \xi^{\perp}_2 = \emptyset$, $ xy \cap \xi_1^{\perp} = z$,
for some point $z \in xy$. Assume also that   $xy$ lies on a line  $l$,  $l(s) = b + s a$, $a,b \in \mathbb{E}^d$, $ s \in \mathbb{R}$. To prove the claim consider  a continuous path $\xi(t)$ along the shortest  arc $[\xi_1\xi_2]$$\subset U_1$, such that $\xi_1=\xi(0)$, $\xi_2 = \xi(1)$ (since $U_1$ is geodesically convex, $[\xi_1\xi_2]\in U_1$).

For any $t \in (0,1)$, the intersection of $\xi(t)^{\perp}$ with $l$, if exists, can be found as $r(t) = b - \frac{\xi(t) \cdot b}{\xi(t) \cdot a} a$. In particular,  $z = r(0)=b - \frac{\xi(0) \cdot b}{\xi(0) \cdot a}  a\in xy$ and $r(1) =b - \frac{\xi(1) \cdot b}{\xi(1) \cdot a}a  \not \in xy$.

Assume that there exists $t_1 \in (0,1]$ such that
$\xi(t_1)^{\perp}$ be parallel to $l$, i.e. $\xi(t_1)\cdot a=0$ (note  that there is at most one such value $t_1$, since $[\xi_1\xi_2]$ and $a^{\perp}$ intersect at most at one point).
Take $\xi(t_1-\delta)$ and consider the point  $r(t_1-\delta)=b - \frac{\xi(t_1-\delta) \cdot b}{\xi(t_1-\delta) \cdot a}  a \not \in xy$ for some $\delta>0$ small enough (if  $\xi(t_1-\delta)$ is not orthogonal to $a$, then line $l$ is not parallel to hyperplane $\xi(t_1-\delta)^{\perp}$, i.e. $l \cap \xi(t_1-\delta)^{\perp} \neq \emptyset$). Since $r(t)$ is continuous on $[0,t_1-\delta]$, there exists $t_0 \in (0,t_1-\delta)$, such that $r(t_0) = x$ (or $y$). This gives $\xi(t_0) \in Y$, which leads to a contradiction.

If $\forall t\in (0,1]$, $\xi(t) \cdot a \neq 0$,  repeat the previous argument with $\delta=0$.
\end{proof}

{\it Step 2}. For any $\xi \in U_1$ the vertices $v_i$ of $P\cap \xi^{\perp}$ are the intersections  of $\xi^{\perp}$ with   lines $l_i\in L$ and  the vertices $\tilde{v}_i $ of $Q\cap \xi^{\perp}$
are the intersections of $\xi^{\perp}$ with
  lines $\tilde{l}_i\in\tilde{L}$. Since
a rigid motion $T_{\xi}$ maps the vertices  of $P\cap \xi^{\perp}$ onto the vertices  of $Q\cap \xi^{\perp}$, there exists  a one to one correspondence between  these sets of vertices. Hence, there also  exists a one to one correspondence between the
lines in $L$ and $\tilde{L}$ (see Figure \ref{sect}).

\begin{figure}[h]
  \centering
  \includegraphics[scale=0.4]{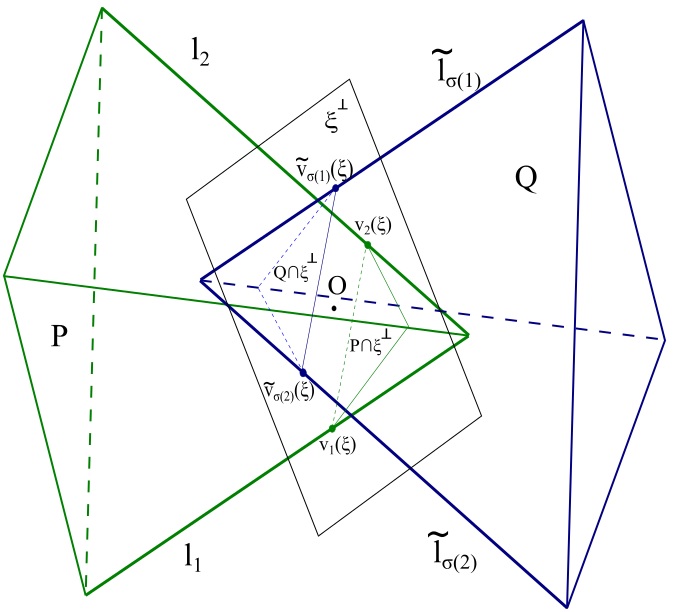}\\
  \caption{Correspondence of lines containing edges through the sections. }\label{sect}
\end{figure}

Let
$L=$$\{l_i\}_{i=1}^k$ and $\tilde{L}=$$\{\tilde{l}_i\}_{i=1}^k$.
Then for every $\xi\in U_1$, there exists at least one  rigid motion $T_{\xi}$ that maps any vertex $v_i(\xi)=\xi^{\perp}\cap l_i$ of $P\cap\xi^{\perp}$  into the vertex $\tilde{v}_j(\xi)=\xi^{\perp}\cap \tilde{l}_{j}$ of $Q\cap\xi^{\perp}$,
$j=j(i,\xi)$.
Hence, there is a permutation $\sigma_{\xi}\in{\mathcal P}_k$ of the set $\{1,\ldots,k\}$ and the map $f_{\xi}: L\to \tilde{L}$ such that $f_{\xi}(l_i)=\tilde{l}_{\sigma_{\xi}(i)}$,  $\sigma_{\xi}(i)=j(i,\xi)$.

We claim that   there exists an open non-empty subset $U_o$ of $U_1$ and a fixed permutation $\sigma=\sigma(U_o)\in{\mathcal P}_k$ such that
\begin{equation}\label{perm}
f_{\xi} (l_i) = \tilde{l}_{\sigma(i)}\qquad \forall\xi \in U_o,\qquad\forall i=1,\ldots,k.
\end{equation}

Take a closed spherical cap with a  non-empty interior $W \subset U_1$.  The set  of all  possible maps $\{f_{\xi}\}_{\{\xi\in W\}}$ is finite. Hence,
$$
W = \bigcup\limits_{\sigma \in {\mathcal P}_k} V_{\sigma},\qquad V_{\sigma} = \{\xi \in W:\,\exists f_{\xi} \,\,\textrm{such that}\,\,f_{\xi}(l_i)=\tilde{l}_{\sigma(i)}\quad\forall i=1,\dots,k\}.
$$

Observe each $V_{\sigma}$ is a closed set (it might be empty). Indeed, let $(\xi_k)_{k=1}^{\infty}$ be a convergent sequence of points of a non-empty $V_{\sigma}$, and let $\lim\limits_{k\to\infty}\xi_k=\xi$.
Then $f_{\xi_k}(l_i)=\tilde{l}_{\sigma(i)}$. In other words,  for the corresponding vertices $v_i(\xi_k)$ and $ \tilde{v}_{\sigma(i)}(\xi_k)$ we have
\begin{equation}\label{vertline}
T_{\xi_k}(v_i(\xi_k))=\tilde{v}_{\sigma(i)}(\xi_k),\qquad v_i(\xi_k)=l_i\cap\xi^{\perp}_k,\quad \tilde{v}_{\sigma(i)}(\xi_k)=\tilde{l}_{\sigma(i)}\cap\xi^{\perp}_k,
\end{equation}
where $\sigma$ is independent of $\xi_k$. Arguing as in the case of projections (see {\it Step 2)} in the proof of Theorem \ref{th1}), we extend the operator $T_{\xi}:\,$$\xi^{\perp}\to\xi^{\perp}$, $\xi\in W$,  to the one acting on the whole space ${\mathbb T}_{\xi}:\,$${\mathbb E^d}\to {\mathbb E^d}$, $\xi\in W$, by the formula
${\mathbb T}_{\xi}(a)=\Phi_{\xi}(a)+b_{\xi}$. Here,
$\Phi_{\xi}\in O(d)$ is defined as $\Phi_{\xi}|_{\xi^{\perp}}=\varphi_{\xi}$, $\Phi_{\xi}(\xi)=\xi$, where $T_{\xi}(a_{\xi})=\varphi_{\xi}(a_{\xi})+b_{\xi}$. Writing (\ref{vertline})
in terms of ${\mathbb T}_{\xi_k}$ we have
$$
{\mathbb T}_{\xi_k}(l_i\cap\xi^{\perp}_k)=  \Phi_{\xi_k}(l_i\cap\xi^{\perp}_k-l_i\cap\xi^{\perp})+\Phi_{\xi_k}(l_i\cap\xi^{\perp})+b_{\xi_k}=\tilde{l}_{\sigma(i)}\cap\xi^{\perp}_k.
$$
Without loss of generality, both polytopes $P$ and $Q$ are located inside a large ball. Hence, the entries $\{a_{ij}^k\}_{k\in \mathbb{N}}$, $i,j =1,2,\ldots,d$, of the matrix corresponding to the transformation $\Phi_{\xi_k}$, and the coordinates $b_j^k$, $j=1,2,\ldots,d$, of the vector $b_{\xi_k}$ are bounded functions of $\xi_k$. By compactness, we can assume that $\{a_{ij}^k\}_{k\in \mathbb{N}}$ and $b_j^k$ are convergent to the corresponding entries of $\tilde{\Phi}_{\xi} = \lim_{k \to \infty} \Phi_{\xi_k}$ and $\tilde{b}_{\xi} = \lim_{k \to \infty} b_{\xi_k}$ respectively. This yields
$$
\tilde{\mathbb{T}}_{\xi}(l_i \cap \xi^{\perp}) = \tilde{\Phi}_{\xi}(l_i \cap \xi^{\perp}) + \tilde{b}_{\xi} = l_{\sigma(i)} \cap \xi^{\perp},
$$
where $\tilde{\mathbb{T}}_{\xi} = \lim_{k \to \infty} \mathbb{T}_{\xi_k}$. Hence, there exists $\tilde{T}_{\xi}$, such that the corresponding $\tilde{f}_{\xi}$ satisfies $\tilde{f}_{\xi}(l_i \cap \xi^{\perp}) = l_{\sigma(i)} \cap \xi^{\perp}$, $\forall i = 1,\ldots,k$. This means that $\xi \in V_{\sigma}$ and $V_{\sigma}$ is a closed.

By the Baire category Theorem (we argue as in {\it Step 2} in the proof of Theorem \ref{th1}), there exists a permutation $\sigma_o$ such that the interior $\textrm{int}(V_{\sigma_o})$ is non-empty.
Hence,   (\ref{perm}) holds  with  $\sigma=\sigma_o$ and $U_o=\textrm{int}(V_{\sigma_o})$.

{\it Step 3}. By Lemma \ref{edges} the set of edges intersected by $\xi^{\perp}$, $\xi\in U_o$, coincides with the set of edges intersected by $\xi^{\perp}$, $\xi\in U_1$.
To show that the corresponding lines in $L$ and $\tilde{L}$  coincide up to a reflection in the origin, we will consider two cases: at least two lines in $L$ are not parallel, or  all lines in $L$ are parallel.

Let $\xi\in U_o$ and let $\sigma$ be as in (\ref{perm}). If at least two lines from $L$,   say $l_1, l_2$,  are not parallel, we
 apply the second part of Lemma \ref{lines} to  the pair of lines $(l_1, l_2)$ with vertices $v_1(\xi)=\xi^{\perp}\cap l_1$, $v_2(\xi)=\xi^{\perp}\cap l_2$ and the corresponding pair  of lines $(\tilde{l}_{\sigma(1)}, \tilde{l}_{\sigma(2)})$$\in \tilde{L}$ with vertices $\tilde{v}_{\sigma(1)}(\xi)=\xi^{\perp}\cap \tilde{l}_{\sigma(1)}$, $\tilde{v}_{\sigma(2)}(\xi)=\xi^{\perp}\cap \tilde{l}_{\sigma(2)}$. We see that the lines coincide up to a reflection in the origin.

%Considering all corresponding pairs of pairwise non-parallel lines $(l_1,l_j)$  and $(\tilde{l}_{\sigma(1)},\tilde{l}_{\sigma(j)})$ (or  $(l_2,l_j)$ and $(\tilde{l}_{\sigma(2)},\tilde{l}_{\sigma(j)})$), $j=3$, $4$, $\dots,k$, with the corresponding vertices, \textcolor[rgb]{1.00,0.00,0.00}{we conclude that all lines from $L$ and $\tilde{L}$ coincide up to a reflection in the origin. }

Assume that $l_1 = \tilde{l}_{\sigma(1)}$ and $l_2 = \tilde{l}_{\sigma(2)}$. Then we choose any other line $l_m$ in $L$, such that $l_m\not\parallel l_1$ (if such $l_m$ does not exists, choose $l_m \not\parallel l_2$). Then we apply the second part of Lemma \ref{lines} to the pair $(l_m, l_1)$.
Recall that the permutation $\sigma$ is the same in all $\xi^{\perp}$, such that $\xi \in U_1$, and we already have that $l_{\sigma(1)} = l_1$. This implies that $\tilde{l}_{\sigma(m)} = l_m$, since none of the lines passes through the origin, i.e. $l_i \neq -l_i$ or $\tilde{l}_j \neq -\tilde{l}_j$ for any $i$ or $j$.
 %If $l_{\sigma(m)} = -l_m$, then $l_{\sigma(1)} = -l_1$, but we have already had $l_{\sigma(1)} = l_1$. Hence, $l_1 = -l_1$, which is not possible due to $O \not \in l_1$. This gives, $l_{\sigma(m)} = l_m$.
We can repeat this argument in the case when $\tilde{l}_{\sigma(1)}=-l_1$ and $\tilde{l}_{\sigma(2)}=-l_2$ to conclude that $\tilde{l}_{\sigma(m)} = -l_m$ for any other line $l_m$ from $L$. Thus,
\begin{equation}\label{shift}
l_j= \tilde{l}_{\sigma(j)} \quad \forall j=1,\ldots,k \quad \textrm{or} \quad l_j= -\tilde{l}_{\sigma(j)} \quad \forall j=1\ldots,k.
\end{equation}
  %for all $j=1$, $\dots, k$, $l_j= \tilde{l}_j$ \textcolor[rgb]{1.00,0.00,0.00}{or $l_j= -\tilde{l}_j$.}

Suppose now that for all $\xi \in U_1$, $\xi^{\perp}$ intersects only parallel  lines in $L$. We claim that (\ref{shift}) holds in this case as well.

Consider a triple of lines $l_p, l_q, l_r \in L$, such that $v_i(\xi) = l_i \cap \xi^{\perp}, (i =p,q,r)$ are vertices of $P \cap \xi^{\perp}$. Notice that in this case the triple doesn't belong to a single two-dimensional plane. Apply Lemma \ref{lines} to the pairs $l_p, l_q$ and $l_p, l_r$, on the open set $\xi \in U_1 \setminus (U_1 \cap Y_{pqr})$. Here $Y_{pqr}$ is a nowhere dense subset obtained from applying Lemma \ref{perpl} to the above triple. Assume that this yields
$$
l_{\sigma(p)} = -l_p + c_{pq}, \quad l_{\sigma(q)} = -l_q + c_{pq}, \quad c_{pq} \in \mathbb{E}^d,
$$
but
$$
l_{\sigma(p)} = l_p + b_{pr},  \quad l_{\sigma(r)} = l_r + b_{pr}, \quad b_{pr} \in \mathbb{E}^d.
$$
Then consider two triangles $v_p(\xi) v_q(\xi) v_r(\xi)$ and $\tilde{v}_{\sigma(p)}(\xi) \tilde{v}_{\sigma(q)}(\xi) \tilde{v}_{\sigma(r)}(\xi)$ in $\xi^{\perp}$ (see Figure \ref{triang}). Notice that, by Lemma \ref{perpl}, $\angle v_q(\xi) v_p(\xi) v_r(\xi) \neq \frac{\pi}{2}$.

\begin{figure}[h]
  \centering
  \includegraphics[scale=0.4]{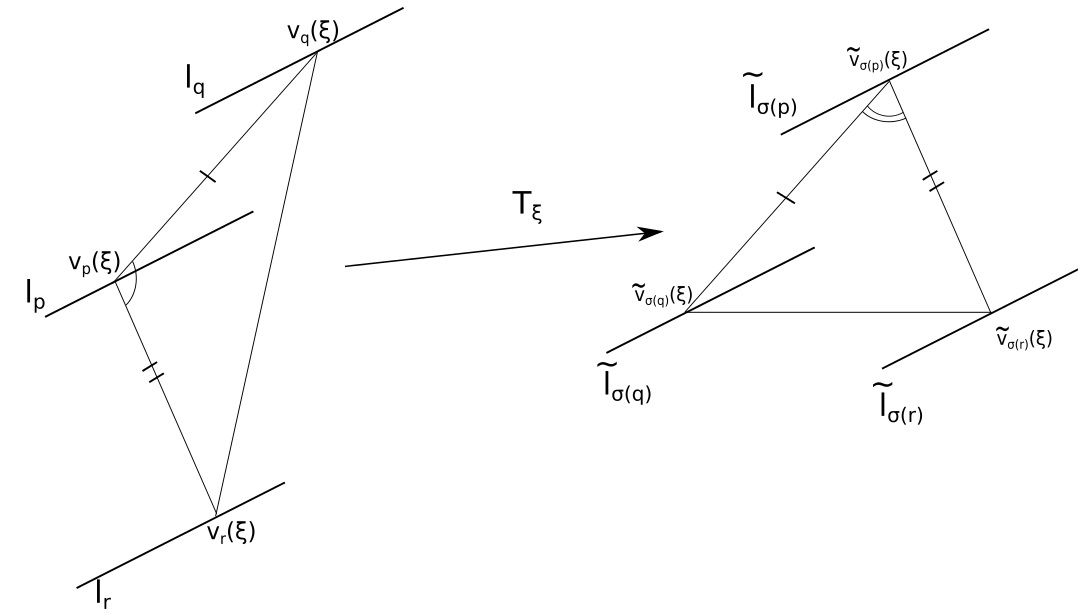}\\
  \caption{$T_{\xi}\left(v_p(\xi) v_q(\xi) v_r(\xi)\right) \neq \tilde{v}_{\sigma(p)}(\xi) \tilde{v}_{\sigma(q)}(\xi) \tilde{v}_{\sigma(r)}(\xi)$}\label{triang}
\end{figure}

On the other hand, since $\angle v_q(\xi) v_p(\xi) v_r(\xi) \neq \angle \tilde{v}_{\sigma(q)}(\xi) \tilde{v}_{\sigma(p)}(\xi) \tilde{v}_{\sigma(r)}(\xi)$, these triangles are not congruent under the fixed permutation $\sigma$, which contradicts the condition of congruency of $P \cap \xi^{\perp}$ and $Q \cap \xi^{\perp}$. Hence, for the triples we have
$$
l_{\sigma(p)} = l_p + b_{pr},  \quad l_{\sigma(q)} = l_q + b_{pr}, \quad l_{\sigma(r)} = l_r + b_{pr}, \quad b_{pr} = b \in \mathbb{E}^d,
$$
or
$$
l_{\sigma(p)} = -l_p + c_{pr},  \quad l_{\sigma(q)} = -l_q + c_{pr}, \quad l_{\sigma(r)} = -l_r + c_{pr}, \quad c_{pr} = b \in \mathbb{E}^d.
$$

Now we can repeat the same argument for any triples of lines in $L$ and $\tilde{L}$. Since $L$ contains a finite number of lines, we exclude a finite number of sets $\{Y_{pqr}\}_{l_p,l_q,l_r \in L}$ obtained from Lemma \ref{perpl}. Recall that each of these sets is closed and nowhere dense on $S^{d-1}$. We conclude that $\tilde{L} = L + b$ or $\tilde{L} = - L + b$. Assume that $\tilde{L} = L + b$, since we can always consider polytope $-Q$ instead of $Q$. We claim that $b=0$.

Consider all edges of $P$ that have a common vertex $q$ with the edge $w_i\subset l_i \in L$, say, $w_p$, $p=1,\dots, s(i)$.

 Consider also  all lines $l_p$ containing $q$, and let $U_{i,p}$ be the corresponding non-empty connected component of $U$ for which $span\{l_i, l_p\} \not \subset\xi^{\perp}$ for all $\xi\in U_{i,p}$ (since the interior of the solid angle $A_q$ with vertex at $q$ and edges $w_1,\ldots, w_{s(i)}$ has dimension $d$, and $\dim \xi^{\perp} = d-1$, this connected component always exists). In particular, $\textrm{dim}(\textrm{span}(l_i, l_p, O))=3$.

We apply the third part of Lemma \ref{lines} to the pairs $(l_i, l_p)$ and $(\tilde{l}_{\sigma(i)}, \tilde{l}_{\sigma(p)})$ and the corresponding vertices $v_i(\xi)=\xi^{\perp}\cap l_i$,
$v_p(\xi)=\xi^{\perp}\cap l_p$ and $\tilde{v}_{\sigma(i)}(\xi)=\xi^{\perp}\cap \tilde{l}_{\sigma(i)}$,
$\tilde{v}_{\sigma(p)}(\xi)=\xi^{\perp}\cap \tilde{l}_{\sigma(p)}$. Here the lines $l_i$ and $\tilde{l}_{\sigma(i)}$ belong to the corresponding sets of lines $L(U_{i,p})$ and
$\tilde{L}(U_{i,p})$, and $\sigma$ is the fixed permutation analogous to the one in (\ref{perm}).

If  there exists $i$, such that for the line $l_i \in L$ Lemma \ref{lines} gives $l_i=\tilde{l}_{\sigma(i)}$, $l_p=\tilde{l}_{\sigma(p)}$, for all
$p=1,\dots, s(i)$, then $b=0$.
Indeed, in this case
 both polytopes $P$ and $Q$  have a common solid angle $A_q$ containing $q$.
Let $H\supset l_i$, $H\cap P=w_i$, be a supporting hyperplane to both $P$ and $Q$ with an inner (with respect to $P$) unit normal $n_H$. We will consider  two cases, $n_H\cdot b>0$ and $n_H\cdot b<0$, and show that both are impossible (by changing $H$ slightly we can assume that $n_H\cdot b\neq 0$).
By the above, $P\subset \textrm{conv}(l_1,\dots,l_k)$ and $Q\subset \textrm{conv}(l_{\sigma(1)},\dots,l_{\sigma(k)})=b+\textrm{conv}(l_1,\dots,l_k)$.
If $n_H\cdot b>0$, then $Q$ should lie in the translated half-space ${\mathcal H}+b$, where  ${\mathcal H}$ is the half-space  with boundary $H$ containing $P$ and $Q$.
Since $Q$  contains $w_i$ this is impossible, unless $b=0$. If $n_H\cdot b<0$, then $Q$ contains $A_q$ and $w_i+b$, which contradicts the convexity of $Q$. Hence, $b=0$.

If for every $i=1,\dots,k$ (for every line $l_i \in L$) there exists $p=p(i)$, $p=1,\dots, s(i)$, for which Lemma \ref{lines} yields $l_i=-\tilde{l}_{\sigma(i)}$, $l_p=-\tilde{l}_{\sigma(p)}$,
 we have  $l_i=-\tilde{l}_{\sigma(i)}$ for all $i=1,\dots, k$.
%Thus, a\textcolor[rgb]{1.00,0.00,0.00}{ll} corresponding lines from  $L$ and $\tilde{L}$ coincide up to a reflection in the origin.

Thus, (\ref{shift}) holds in the case when all lines in $L$ are parallel as well.

{\it Step 4}.
We proved that for any $\xi \in U_i$ the corresponding lines $L$ and $\tilde{L}$ that are intersected by $\xi^{\perp}$ satisfy (\ref{shift}). This implies that for any vertex $v_i(\xi) \in P \cap \xi^{\perp}$ there exists a vertex $v_{\sigma(i)_{\xi}}(\xi) \in Q\cap \xi^{\perp}$ such that
$$
v_{i}(\xi) = \tilde{v}_{\sigma(i)}(\xi) \quad \forall i=1,\ldots,k \quad \textrm{or} \quad v_{i}(\xi) = -\tilde{v}_{\sigma(i)}(\xi) \quad \forall i=1,\ldots,k.
 $$
%\textcolor[rgb]{1.00,0.00,0.00}{Also, observe that if two vertices $v_i(\xi), v_j(\xi) \in P_{\xi}$ are connected by an edge, so are the corresponding vertices $v_{\sigma(i)}(\xi), v_{\sigma(j)}(\xi) \in Q_{\xi}$. }
Since any convex polytope is a convex hull of its vertices, we conclude that
$$
Q\cap \xi^{\perp} = P\cap \xi^{\perp} \quad \forall \xi \in U \quad \textrm{or} \quad Q\cap \xi^{\perp} = -P\cap \xi^{\perp} \quad \forall \xi \in U.
$$
This implies that for all $\xi\in U\setminus (Y\cup Y_1)$ the radial functions of $P$ and $Q$ satisfy $\rho_P(u) = \rho_Q(u)$ for all $u \in \xi^{\perp}$ or
$\rho_P(-u) = \rho_Q(u)$ for all $u \in \xi^{\perp}$.
Since the  radial functions are continuous, we can prove that the same holds for any $\xi \in Y\cup Y_1$, and hence for all $\xi\in S^{d-1}$.  Finally, we can apply Theorem \ref{hedgehog} with $j=d-1$, $a_{\alpha}=0$,  to conclude that $\rho_P(u) = \rho_Q(u)$ for all $u\in S^{d-1}$ or $\rho_P(u) =\rho_Q(-u)$ for all $u\in S^{d-1}$(or see Lemma 1 in \cite{R2}) . Theorem  \ref{th2} is proved in the case $k=d-1$.

\subsection{Proof of Theorem \ref{th2} in the case $2\leq k < d-1$}
We use  induction on $k$.
Let $H$ be any $(k+1)$-dimensional subspace of ${\mathbb E^d}$. We apply the result for $d=k+1$ to the bodies $P\cap H$ and $Q\cap H$ and their sections
$(P\cap H)\cap J$ and $(Q\cap H)\cap J$ for all $k$-dimensional subspaces $J\subset H$.
We obtain that $P\cap H= Q\cap H$ or $P\cap H = -Q\cap H$ for all $H$. We proceed and the result follows after finitely many steps.
The proof of Theorem \ref{th2}  is finished.

\end{document}